\documentclass{amsart}

\usepackage{amssymb, amscd, latexsym, graphicx, psfrag}
\usepackage[all]{xy}

\newtheorem{dummy}{dummy}[section]
\newtheorem{lemma}[dummy]{Lemma}
\newtheorem{theorem}[dummy]{Theorem}

\newtheorem{proposition}[dummy]{Proposition}
\newtheorem{example}[dummy]{Example}

\theoremstyle{definition}
\newtheorem{remark}[dummy]{Remark}

\newcommand{\Ga}{\Gamma}

\newcommand{\Si}{\Sigma}

\newcommand{\ep}{\epsilon}
\newcommand{\la}{\lambda}
\newcommand{\si}{\sigma}

\newcommand{\cD}{\mathcal{D}}
\newcommand{\cF}{\mathcal{F}}
\newcommand{\cG}{\mathcal{G}}

\newcommand{\cL}{\mathcal{L}}
\newcommand{\cO}{\mathcal{O}}
\newcommand{\cQ}{\mathcal{Q}}
\newcommand{\cS}{\mathcal{S}}
\newcommand{\cX}{\mathcal{X}}

\newcommand{\cW}{\mathcal{W}}

\newcommand{\bC}{ {\mathbb{C}} }

\newcommand{\bP}{ {\mathbb{P}} }
\newcommand{\bQ}{ {\mathbb{Q}} }
\newcommand{\bR}{ {\mathbb{R}} }
\newcommand{\bZ}{ {\mathbb{Z}} }

\newcommand{\bSi}{\mathbf{\Si}}
\newcommand{\bc}{\mathbf{c}}
\newcommand{\bt}{\mathbf{t}}

\newcommand{\Hom}{\mathrm{Hom}}
\newcommand{\can}{\mathrm{can}}
\newcommand{\coker}{\mathrm{coker}}
\newcommand{\Ext}{\mathrm{Ext}}
\newcommand{\Spec}{\mathrm{Spec}}

\newcommand{\fofr}{\mathfrak{or}}
\newcommand{\fm}{\mathfrak{m}}

\newcommand{\tM}{\widetilde{M}}
\newcommand{\tN}{\widetilde{N}}
\newcommand{\tT}{\widetilde{T}}

\newcommand{\tb}{\widetilde{b}}
\newcommand{\tp}{\widetilde{p}}
\newcommand{\tit}{\widetilde{t}}

\newcommand{\vc}{ {\vec{c}} }

\newcommand{\naive}{\mathit{naive}} 
\newcommand{\lra}{\longrightarrow}
\newcommand{\LS}{\Lambda_\Si}
\newcommand{\LbS}{\Lambda_{\bSi}}
\newcommand{\Perf}{\mathcal{P}\mathrm{erf}}
\newcommand{\Sh}{\mathit{Sh}}
\newcommand{\tri}{\triangle}

\newcommand{\Coh}{Coh}
\newcommand{\fin}{\mathit{fin}}
\newcommand{\ltr}{\langle \Theta \rangle}
\newcommand{\ltrp}{\langle\Theta'\rangle}

\begin{document}

\title[CCC and Fourier-Mukai Transforms]{The Coherent-Constructible Correspondence and Fourier-Mukai Transforms}

\author{Bohan Fang}
\address{Bohan Fang, Department of Mathematics, Columbia University,
2990 Broadway, New York, NY 10027}
\email{b-fang@math.columbia.edu}

\author{Chiu-Chu Melissa Liu}
\address{Chiu-Chu Melissa Liu, Department of Mathematics, Columbia University,
2990 Broadway, New York, NY 10027} \email{ccliu@math.columbia.edu}

\author{David Treumann}
\address{David Treumann, Department of Mathematics, Northwestern University,
2033 Sheridan Road, Evanston, IL 60208}
\email{treumann@math.northwestern.edu}

\author{Eric Zaslow}
\address{Eric Zaslow, Department of Mathematics, Northwestern University,
2033 Sheridan Road, Evanston, IL 60208}
\email{zaslow@math.northwestern.edu}

\dedicatory{Dedicated to Professor Loo-Keng Hua on the occasion of his 100th birthday}

\begin{abstract}
In \cite{Ka}, as evidence for his conjecture in birational log
geometry, Kawamata constructed a family of derived equivalences
between toric orbifolds.  In \cite{stack} we showed that the derived
category of a toric orbifold is naturally identified with a category
of polyhedrally-constructible sheaves on $\bR^n$.  In this paper we
investigate and reprove some of Kawamata's results from this perspective.
\end{abstract}

\maketitle

\tableofcontents

\section{Introduction}

\subsection{Coherent-constructible correspondence}

The coherent-constructible correspondence (CCC), defined in \cite{more, Tr} 
and first described by Bondal \cite{Bo}, is an equivalence between a category of coherent sheaves on a toric $n$-fold and a category of constructible sheaves on a compact $n$-torus $(S^1)^n$.  An equivariant version of this correspondence can be
viewed as a ``categorification'' of Morelli's description of the
equivariant K-theory of a toric variety in terms of a polytope
algebra \cite{Mo}.  Generalizing the familiar correspondence in toric
geometry between ample equivariant line bundles and moment polytopes,
the equivariant CCC provides an equivalence between torus-equivariant
coherent sheaves on a toric n-fold and constructible sheaves on
$\bR^n$, the universal cover of $(S^1)^n$.

The CCC was extended to toric Deligne-Mumford (DM) stacks in \cite{stack}.  
Toric DM stacks were defined by Borisov-Chen-Smith
in terms of stacky fans \cite{BCS}. In this paper we consider
toric orbifolds, which are toric DM stacks with generically
trivial stabilizers. Let $\cX_\bSi$ be a complete
toric orbifold defined by a stacky fan $\bSi=(N,\Si,\beta)$,
where $N\cong \bZ^n$, $\Si$ is a simplicial fan in 
$N_\bR:=N\otimes _\bZ\bR \cong \bR^n$.  It contains the torus
$T\cong (\bC^*)^n$ as a dense open subset.
The CCC for the toric orbifold $\cX_\bSi$ is the following quasi-equivalence of 
triangulated dg categories:
\begin{equation}\label{eqn:ccc}
\kappa_\bSi:\Perf_T(\cX_\bSi)\stackrel{\sim}{\longrightarrow} Sh_{cc}(M_\bR;\LbS)
\end{equation}
where $\Perf_T(\cX_\bSi)$ is the category of $T$-equivariant
perfect complexes on $\cX_\bSi$, and 
$Sh_{cc}(M_\bR;L_\bSi)$ is a category of constructible
sheaves on $M_\bR$ (the dual space of $N_\bR$)
characterized by a conical Lagrangian $\LbS \in T^*M_\bR$
determined by the stacky fan $\bSi$. (The
precise definitions of the categories in \eqref{eqn:ccc} will
be given in Section \ref{sec:ccc}.)  
Moreover, the functor $\kappa_{\bSi}$ is monoidal with respect
to the tensor product of coherent sheaves
on $\cX_\bSi$ and the convolution product
of constructible sheaves on $M_\bR$. Please note that since toric orbifolds are {\em smooth} DM stacks, 
the category $\Perf_T(\cX)$ is the same as the category $\Coh_T(\cX)$, and we will use both notations interchangeably throughout the paper.

\subsection{Fourier-Mukai Transforms}
The coarse moduli space of the toric orbifold
$\cX_{\bSi}$ is the toric variety $X_\Si$ defined
by the simplicial fan $\Si$. The toric orbifold
$\cX_\bSi$ is the DM stack associated to a log pair
$(X_\Si,B)$ in the sense of Kawamata \cite[Definition 2.1]{Ka}.
Kawamata considered pairs $(X,B)$ of varieties and $\bQ$-divisors
which have smooth local coverings. For such a pair he
associated a DM stack $\cX$, such that
$p^*(K_X+B)=K_{\cX}$, where $p:\cX\to X$ is the canonical
map to the coarse moduli space. 
He conjectured that if there is an equivalence
of log canonical divisors between birationally
equivalent pairs, then there is an equivalence of
derived categories of the associated DM stacks \cite[Conjecture 2.2]{Ka}. 
He proved his conjecture for quasi-smooth toroidal pairs.
Here we briefly describe his results in the toric case; please
see \cite[Theorem 4.2]{Ka} for the precise statements, in 
the toroidal case (which includes the toric case as a special
case). Let $\cX_1$ and $\cX_2$ be toric orbifolds
associated to projective toric log pairs $(X_1,B)$ and $(X_2,C)$, respectively.
Suppose that $\cX_1$ and $\cX_2$  are $K$-equivalent \cite{Wa} in the sense that
there exists a toric orbifold $\cW$ and proper
birational morphisms $\mu_i:\cW\to \cX_i$ of toric orbifolds
such that $\mu_1^*K_{\cX_1}=\mu_2^*K_{\cX_2}$. Then
there is an equivalence of triangulated categories:
$D^bCoh(\cX_2)\cong D^bCoh(\cX_1)$.
Indeed, Kawamata proved that if $\mu_1^*K_{\cX_1}\geq \mu_2^*K_{\cX_2}$
and the birational map $f:X_1 \dashrightarrow X_2$ is (1) the identity, (2) a divisorial
contraction, (3) the inverse of a divisorial contraction, or (4) a flip, 
then the Fourier-Mukai functor $F'_{12}=\mu_{1*}\circ \mu_2^*$ 
\footnote{We use $F_{12}'$ instead of $F_{12}$ since the notation without prime is reserved for the induced functor on constructible sheaves.}
is fully faithful; it is an equivalence
when $\mu_1^*K_{\cX_1} =\mu_2^* K_{\cX_2}$. The McKay correspondence
for abelian quotient singularities is a special case of (2) or (3). 

\subsection{CCC and Fourier-Mukai transforms}
It is natural to expect that Kawamata's theorem \cite[Theorem 4.2]{Ka} holds 
in the equivariant setting, and that one can use CCC 
to give an elementary proof. The following square 
of functors commutes up to natural isomorphisms 
$$
\begin{CD}
\Perf_T(\cX_2) & @>{\kappa_2}>>  & Sh_{cc}(M_\bR;\Lambda_{\bSi_2})\\
@V{F'_{12}}VV & &  @V{F_{12}}VV \\
\Perf_T(\cX_1) & @>{\kappa_1}>> &  Sh_{cc}(M_\bR;\Lambda_{\bSi_1})
\end{CD}
$$
where $\kappa_i$ are quasi-equivalences. So it suffices to prove
that $F_{12}$ is cohomologically full and faithful in various
cases (1)--(4). In this paper, we provide such elementary proofs
in cases (1), (2), and (3).   
We describe $F_{12}$ and $F_{12}'$ explicitly in terms of theta sheaves (see
Section \ref{sec:theta}) which are building blocks
of the CCC. Our proofs do not rely on the vanishing theorems
in \cite{KMM}.

Note that $F_{12}$ and $F'_{12}$ do not preserve
the monoidal structures.

\begin{remark}
Although we are restricting to \emph{toric orbifolds}, i.e. toric DM stacks with a generically trivial stabilizer, 
Proposition 3.1 of \cite{stack} shows that for any toric DM stack $\cX$, the dg category $\Perf_{\mathcal T}(\cX)\cong \Perf_T(\cX^\mathrm{rig})$ 
where $\mathcal T$ is the DM torus acting on $\cX$, and the orbifold $\cX^{\mathrm{rig}}$ is the rigidification of $\cX$.
(See \cite{FMN} for definitions.) 
Thus the result of this paper implies the functor $\Perf_{\mathcal T_2}(\cX_2)\to \Perf_{\mathcal T_1}(\cX_1)$ is a 
quasi-embedding in cases (1), (2) and (3). 
\end{remark}

\subsection{A simple example: McKay correspondence for the $A_1$-singularity}
$X_2=\bC^2/\bZ_2$ is the $A_1$-singularity,
$X_1=\cO_{\bP^1}(-2)$ is its crepant resolution, and
$\cX_1=X_1$ and $\cX_2=[\bC^2/\bZ_2]$ are the canonical
toric orbifolds associated to $X_1$ and $X_2$ respectively.
In this case both $F'_{12}$ and $F'_{21}$ are equivalences, as shown in Figure \ref{fig:A1}.
\begin{figure}[h]
\psfrag{F12}{\small $F_{12}$}
\psfrag{F21}{\small $F_{21}$}
\includegraphics[scale=0.5]{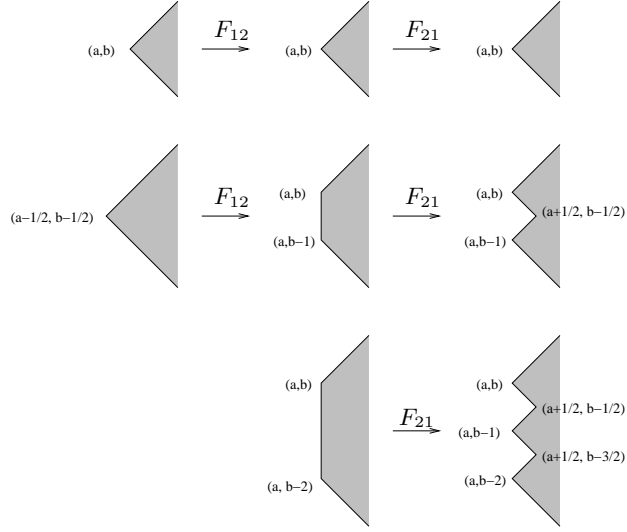}
\caption{$a, b$ are integers. The sheaves are costandard sheaves supported on the shaded area (see Section \ref{sec:constructible} for the definition).}
\label{fig:A1}
\end{figure}

\subsection{Outline}
In Section \ref{sec:definition}, we give a brief introduction to toric orbifolds.
In Section \ref{sec:ccc}, we give a leisure exposition
of the CCC of toric orbifolds; we also state the CCC for toric varieties.
In Section \ref{sec:FM}, we elaborate
Kawamata's theorem in the equivariant setting
from the perspective of constructible sheaves.

\subsection*{Acknowledgments} We thank H.-H. Tseng for bringing \cite[Theorem 4.2]{Ka}
to our attention. We thank Y. Kawamata for his helpful comments on a draft of this paper.

\section{Toric Orbifolds} 
\label{sec:definition}

In \cite{BCS}, Borisov-Chen-Smith introduced toric DM stacks.
In this paper we will consider the case of \emph{toric orbifolds}. A toric
orbifold is a toric DM stack with trivial generic stabilizer.

\subsection{The Stacky Fan}
Let $N\cong \bZ^n$ be a free abelian group,  and let $\Si$ be a simplicial fan in $N_\bR :=N\otimes_{\bZ}\bR \cong \bR^n$. The pair $(N,\Si)$ defines a simplicial toric variety $X_\Si$ of dimension $n$ (see \cite{Fu}). Let $\Si(1)=\{\rho_1,\ldots, \rho_r\}$
be the set of 1-dimensional cones in the fan $\Si$, and let $v_i\in N$ be the unique generator of the semigroup $\rho_i\cap N$, so that $\rho_i\cap N=\bZ_{\geq 0}v_i$. 

A \emph{stacky fan} $\bSi$ is defined as the data $(N,\Si,\beta)$, where 
$$
\beta:\tN:=\oplus_{i=1}^r \bZ \tb_i\cong \bZ^r \to N\cong\bZ^n
$$ 
is a group homomorphism sending $\tb_i$ to $b_i = n_i v_i\in \bZ_{>0}v_i$.
If $n_i=1$ for $i=1,\ldots,r$ then the corresponding map is denoted by $\beta_\can$. 
We assume that $\{ v_1,\ldots, v_r\}$ span $N_\bR$, which implies the cokernel of 
$\beta: \tN \to N$ is finite. 

\begin{example}
\label{ex:fan1}
$N=\bZ^2$, $\tN=\bZ^3$. The fan $\Si$ and the map $\beta$ are shown in Figure \ref{fig:fan1}.
\begin{figure}[h]
\psfrag{S}{}
\psfrag{BETABETABETABETABETA}{\small $\begin{CD}
\tN=\bZ^3 @>{\beta=\beta_\can=\left[\begin{array}{ccc}1&-1&0\\0&-2&1\end{array}\right]} >>
N=\bZ^2.
\end{CD}$}
\includegraphics[scale=0.5]{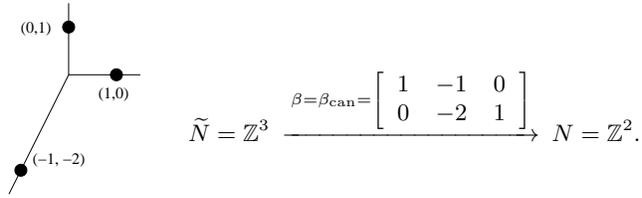}
\caption{The stacky fan $\bSi=(N,\Si,\beta)$, with $\Si$ and $\beta$ shown above. The dots are $b_1=(1,0)$,
$b_2=(-1,-2)$, and $b_3=(0,1)$.}
\label{fig:fan1}
\end{figure}
\end{example}

We next consider a 1-dimensional example in which $\beta\neq \beta_\can$.
\begin{example}
\label{ex:fan2}
$N=\bZ$, $\tN=\bZ$. The fan $\Si$ and the map $\beta$ are shown in Figure \ref{fig:fan2}.
\begin{figure}[h]
\psfrag{BETABETABETABETABETA}{\small $
\begin{CD}
\tN=\bZ @> {\beta=[\begin{array}{cc}3&-1\end{array} ]}>> N=\bZ.
\end{CD}
$}
\includegraphics[scale=0.5]{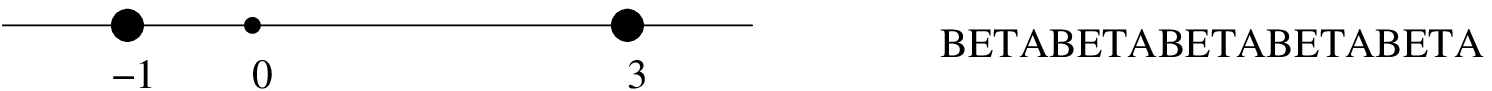}
\caption{The stacky fan $\bSi=(N,\Si,\beta)$, with $\Si$ and $\beta$ shown above. The larger dots are $b_1=3$ and $b_2=-1$.}
\label{fig:fan2}
 \end{figure}
\end{example}

\subsection{Construction of the Toric Orbifold}
Let $M=\Hom(N;\bZ)$ be the dual lattice of $N$, and let
$\tM=\Hom(M;\bZ)$ be the dual lattice of $\tN$. Since
$\beta: \tN\to N$ has a finite cokernel,  the dual map $\beta^*:M\to \tM$ is injective. Applying $\Hom(-,\bC^*)$ to the following short exact sequence 
$$
0\to M\stackrel{\beta^*}{\lra} \tM \stackrel{\beta^\vee}{\lra} \coker(\beta^*)\to 0,
$$
we obtain another short exact sequence
$$
1\to G_{\bSi}\to \tT \to T\to 1,
$$
where $G_{\bSi}:=\Hom(\coker(\beta^*),\bC^*)$ is isomorphic to the direct
product of $(\bC^*)^{r-n}$ and a finite abelian group, and
$$
\tT=\Hom(\tM,\bC^*)\cong (\bC^*)^r, \quad
T=\Hom(M,\bC^*)\cong (\bC^*)^n.
$$

The torus $\tT\cong (\bC^*)^r$ acts on $\bC^r=\Spec\bC[z_1,\ldots, z_r]$.
Let $I_\Si \subset \bC[z_1,\ldots, z_r]$ be
the ideal generated by $\{ \prod_{\rho_i\not \subset \si} z_i\mid \sigma\in \Si\}$.
(Note that the ideal $I_\Si$ depends on $N$ and $\Si$ but not on $\beta$.)
Let $Z(I_\Si)\subset \bC^r$ be the closed subscheme defined by $I_\Si$,
and let $U_\Si=\bC^r-Z(I_\Si)$, which is a Zariski open set of $\bC^r$.
We define $\cX_\bSi$ to be the quotient stack:
$$
\cX_\bSi = [U_\Si /G_\bSi ].
$$
The simplicial toric variety $X_\Si$ defined by 
$\Si$ can be identified with the geometric quotient
$$
X_\Si = U_\Si/G_{\bSi^\can},
$$
where $\bSi^\can=(N,\Si,\beta_\can)$.  
We have a 2-cartesian diagram
$$
\begin{CD}
U_\Si @>{\tp}>> U_\Si\\
@VVV @ VVV\\
\cX_\bSi @>{p}>> X_\Si,
\end{CD}
$$
where
\begin{equation}\label{eqn:tp}
\tp(z_1,\ldots, z_r) = (z_1^{n_1}, \ldots, z_r^{n_r}).
\end{equation}

We have the following properties regarding to $\cX_\bSi$:
\begin{itemize}
\item $\cX_{\bSi}$ is an orbifold, i.e. it is a smooth DM stack with generically trivial
stabilizers.
\item There is an open dense embedding 
$T=\tT/G_\bSi \hookrightarrow \cX_\bSi=[U_\Si/G_\bSi]$, and
the action on $T$ on itself extends to a $T$-action on $\cX_\bSi$.
\item The coarse moduli space of the orbifold $\cX_\bSi$ is
the simplicial toric variety $X_\Si$. The projection $p:\cX_\bSi\to X_\Si$
restricts to the identity map from $T$ to itself.
\end{itemize}

\begin{example}[$\bP(1,1,2)$: Example \ref{ex:fan1} continued] 
\label{ex:stack1}
The stacky fan $\bSi=(N,\Si,\beta)$ is defined as in Example \ref{ex:fan1}. Taking $\Hom(-,\bC^*)$ of the following exact sequence
$$
\begin{CD}
0\to M=(\bZ^2)^* @>{\beta^*=\left[\begin{array}{cc} 1 & 0\\ -1& -2\\ 0&1\end{array}\right]}>>
\tM=(\bZ^3)^* @>{\beta^\vee=\left[\begin{array}{ccc}1 & 1 & 2 \end{array}\right]}>> \bZ\to 0.
\end{CD}
$$
produces
$$
1 \to  G_\bSi=\bC^*  \lra  \tT=(\bC^*)^3 \lra  T=(\bC^*)^2 \to  1.
$$
The map $G_\bSi\to \tT$ is given by  $\la \mapsto (\la,\la,\la^2)$. Therefore, the toric orbifold $\cX_\bSi$ is defined as
$$
\cX_\bSi=\left[(\bC^3-\{(0,0,0)\})/\bC^* \right] =\bP(1,1,2).
$$
where $\bC^*$ acts on $\bC^3$ by $\la\cdot (z_1,z_2,z_3)
=(\la z_1,\la z_2,\la^2 z_3)$. The corresponding coarse moduli space is the following simplicial toric variety
$$
X_\Si = (\bC^3-\{(0,0,0)\})/\bC^* =\mathbf{P}(1,1,2).
$$
It has a unique singularity at $[0,0,1]$. 

\end{example}

\begin{example}[$\bP(1,3)$: Example \ref{ex:fan2} continued]
\label{ex:stack2}
The stacky fan $\bSi=(N,\Si,\beta)$ is defined as in Example \ref{ex:fan2}. 
Then $\beta_\can=[1\; -1]:\bZ^2\to \bZ$. There is a commutative diagram
\begin{equation}\label{eqn:quasi}
\xymatrix{
1 \ar[r] & G_\bSi=\bC^* \ar[r]^\phi \ar[d]^{\hat{p}} 
& \tT=(\bC^*)^2 \ar[r]^\pi \ar[d]^{\tp} 
& T=\bC^* \ar[r] \ar[d]^{p} & 1\\
1 \ar[r] & G_{\bSi_\can}=\bC^* \ar[r]^{\phi_\can} 
& \tT=(\bC^*)^2 \ar[r]^{\pi_\can} & T=\bC^* \ar[r] &1
}
\end{equation}
where the rows are short exact sequences of abelian groups.
The arrows are group homomorphisms given explicitly as follows:
$$
\phi(\la)=(\la,\la^3),\quad \pi(\tit_1,\tit_2)= \tit_1^3 t_2^{-1},\quad
\phi_\can(\la)=(\la,\la),\quad \pi_\can(\tit_1,\tit_2)= \tit_1 \tit_2^{-1},
$$
$$
\hat{p}(\la)=\la^3,\quad \tp(\tit_1,\tit_2)=(\tit_1^3, \tit_2),\quad 
p(t)= t.
$$

The toric orbifold defined by $\bSi$ is a weighted projective line:
$$
\cX_\bSi=\left[(\bC^2-\{(0,0)\})/\bC^* \right] =\bP(1,3),
$$
where $\bC^*$ acts on $\bC^2$ by $\la\cdot (z_1,z_2)
=(\la z_1,\la^3 z_2)$. The coarse moduli space is the projective line:
$$
X_\Si=(\bC^2-\{(0,0)\})/\bC^*,
$$ 
where $\bC^*$ acts on $\bC^2$ by 
$\la\cdot (z_1,z_2)=(\la z_1, \la z_2)$.
\end{example}

We say $X_\bSi$ is a {\em complete} toric orbifold
if $\Si$ is a complete fan in $N_\bR$, or equivalently,
the coarse moduli space $X_\Si$ is a complete toric variety.
For example,   $\bP(1,1,2)$ and $\bP(1,3)$ are complete
toric orbifolds.

\subsection{Divisors and line bundles}
Let $\widetilde{D}_i\subset U_\Si$ be the divisor defined by $z_i=0$, and
let $D_i$  and $\cD_i$ denote the  corresponding $T$ divisors
in $X_\Si$ and $\cX_\bSi$, respectively. Let $p:\cX_\bSi\to X_\Si$
be the canonical map to the coarse moduli space. By \eqref{eqn:tp}, 
$p^*D_i=n_i \cD_i$.

In general, $D_i$ is a $\bQ$-Cartier divisor: there exists some
positive integer $k$ such that $\cO_{X_\Si}(kD_i)$ is a line bundle on $X_\Si$. On
the other hand, $\cO_{\cX_\bSi}(\cD_i)$ is always a line bundle on the toric
orbifold $\cX_\Si$. Indeed, any $T$-equivariant line bundle on $\cX_\bSi$ is of the form
$$
\cL_\vc =\cO_{\cX_\bSi}(\sum_{i=1}^r c_i \cD_i), \quad
\vc=(c_1,\ldots, c_r)\in \bZ^r.
$$

\subsection{Relation to log pairs}\label{sec:log-pair}

In \cite{Ka}, Kawamata considers pairs of varieties with $\bQ$-divisors
which have local covering by smooth varieties, and
associates a Deligne-Mumford stack to such a pair \cite[Definition 2.1]{Ka}.
We now relate toric orbifolds to the Deligne-Mumford stacks in 
\cite[Definition 2.1]{Ka}.

Let $\cX_\bSi$ be the toric orbifold defined by a stacky fan $\bSi=(N,\Si,\beta)$, and
let $X_\Si$ be the simplicial toric variety defined by the simplicial fan $\Si\subset N_\bR$.
Let $p:\cX_\bSi\to X_\Si$ be the canonical map to the coarse moduli space.
Let $n_i \in \bZ_{>0}$ be defined as before. Define a $\bQ$-divisor
$$
B=\sum_{i=1}^r(1-\frac{1}{n_i})D_i
$$
on $X$. Then the pair $(X_\Si,B)$ satisfies the condition in \cite[Definition 2.1]{Ka}, and
the associated Deligne-Mumford stack to this pair is exactly $\cX_\bSi$. 
Let $K_{\cX_\bSi}$ and $K_{X_\Si}$ denote the canonical divisors
on $\cX_\bSi$ and $X_\Si$, respectively. We have the following identities:
\begin{equation}\label{eqn:canonical-divisor}
K_{X_\Si}= -\sum_{i=1}^r D_i,\quad
K_{X_\Si}+ B=- \sum_{i=1}^r \frac{1}{n_i}D_i,
\quad p^*(K_{X_\Si}+B)= -\sum_{i=1}^r \cD_i = K_{\cX_\bSi}.
\end{equation}
In particular, when $n_1=\cdots =n_r=1$, $B=0$, and the Deligne-Mumford
associated to the pair $(X_\Si,0)$ is $\cX_{\bSi_\can}$.

\section{Coherent-Constructible Correspondence}
\label{sec:ccc}

The coherent-constructible correspondence relates
equivariant coherent sheaves on a toric orbifold of
dimension $n$ to certain constructible sheaves on 
a real vector space of dimension $n$. Before
we give the precise statement of the coherent-constructible 
correspondence, we need to review some definitions.

We will use the language of dg categories throughout.
If $C$ is a dg category, then $hom(x,y)$ denotes
the chain complex of homomorphisms between objects $x$ and $y$ of
$C$.  We will continue to use $\Hom(x,y)$ to denote hom sets in non-dg settings.  We will regard the differentials in all
chain complexes as having degree $+1$, i.e. $d:K^i \to K^{i+1}$.  If
$K$ is a chain complex (of vector spaces or sheaves, usually) then
$h^i(K)$ will denote its $i$th cohomology object. If $C$ is a dg category, then $Tr(C)$ denotes  the
triangulated dg category generated by $C$, and
$D(C)$ denotes the cohomology category $H(Tr(C)).$ The triangulated category
$H(Tr(C))$ is sometimes called the \emph{derived category} of $C$.

\subsection{Coherent and quasicoherent sheaves on toric orbifolds}
We refer to \cite[Definition 7.18]{Vi} for the definitions of
quasicoherent sheaves, coherent sheaves, and vector bundles on a general
Deligne-Mumford stack.  If $\cX$ is a Deligne-Mumford stack, let $\cQ(\cX)^\naive$ denote the dg category of bounded complexes of quasicoherent sheaves on $\cX$, and let $\cQ(\cX)$ denote the localization of this category with respect to acyclic complexes.
We use $\Perf(\cX) \subset \cQ(\cX)$ to denote the full dg subcategories
consisting of \emph{perfect} objects---that is, objects which are quasi-isomorphic to bounded complexes of vector bundles.

We now spell out the above definitions for a toric orbifold $\cX_\bSi=[U_\Si/G_\bSi].$ By \cite[Example 7.21]{Vi},
the category of coherent sheaves on $\cX_\bSi$ is equivalent to the category of $G_\bSi$-equivariant coherent sheaves on $U_\Si$.
Similarly, the category of quasicoherent sheaves on $\cX_\bSi$ is equivalent to the category of
$G_\bSi$-equivariant quasicoherent sheaves on $U_\Si$. Therefore,
\begin{equation}
\label{eqn:nonequiv}
\cQ(\cX_\bSi) = \cQ_{G_\bSi}(U_\Si),\quad \Perf(\cX_\bSi)=\Perf_{G_\bSi}(U_\Si).
\end{equation}
We define the category of $T$-equivariant coherent (resp. quasicoherent) 
sheaves on $\cX$ to be equivalent to the category of $\tT$-equivariant coherent (resp. quasicoherent) 
sheaves on $U$: 
\begin{equation}
\label{eqn:equiv}
\cQ_{T}(\cX_\bSi) = \cQ_{\tT}(U_\Si),\quad \Perf_{T}(\cX_\bSi)=\Perf_{\tT}(U_\Si).
\end{equation}
There is a monoidal product structure $\otimes$ on these various dg categories of sheaves on $\cX_\bSi$, simply given by the tensor product of quasi-coherent sheaves on $U_\Si$.

\subsection{Constructible sheaves}
\label{sec:constructible}

We refer to \cite{KS} for the microlocal theory of sheaves. If $X$
is a topological space we let $\Sh(X)$ denote the dg category of
chain complexes of sheaves of $\bC$-vector spaces on $X$,
localized with respect to acyclic complexes (see \cite{Dr} for localizations
of dg categories).  If $X$ is a
real-analytic manifold, $\Sh_c(X)$ denotes the full subcategory of
$\Sh(X)$ of objects whose cohomology sheaves are bounded and constructible with
respect to a real-analytic stratification of $X$. Denote by
$\Sh_{cc}(X) \subset \Sh_c(X)$ the
full subcategory of objects which have compact support. We use
$D_c(X)$ and $D_{cc}(X)$ to denote the derived categories $D(Sh_c(X))$
and $D(Sh_{cc}(X))$ respectively.

The \emph{standard constructible sheaf} on the submanifold $i_{Y}:
Y\hookrightarrow X$  is defined as the push-forward of the constant
sheaf on $Y$, i.e. $i_{Y*} \bC_Y,$ as an object in $Sh_c(X)$. The
Verdier duality functor $\cD: Sh_c^\circ(X)\to \Sh_c(X)$ takes
$i_{Y*}\bC_Y$ to the \emph{costandard constructible sheaf} on $X$.
We know $\cD(i_{Y*}\bC_Y)=i_{Y!}\cD(\bC_Y)=i_{Y!} \omega_Y$. Here
$\omega_Y=\cD(\bC_Y)=\fofr_Y[\dim Y]$, where $\fofr_Y$ is the
orientation sheaf of $Y$ (with respect to the base ring $\bC$).

We denote the singular support of a complex of sheaves $F$ by
$SS(F) \subset T^*X$.   If $X$ is a real-analytic manifold and
$\Lambda \subset T^*X$ is an $\bR_{\geq 0}$-invariant Lagrangian
subvariety, then $\Sh_c(X;\Lambda)$ (resp. $\Sh_{cc}(X;\Lambda)$)
denotes the full subcategory of $\Sh_c(X)$ (resp. $\Sh_{cc}(X)$)
whose objects have singular support in $\Lambda$. For any open subset $U\subset X$, the singular support of the associated standard and costandard sheaves are given by the following theorem.
\begin{theorem}[Schmid-Vilonen]
$$
SS(i_!\omega_U) =  \lim_{\ep\to 0^+} \Ga_{-\ep d\log m},\quad\quad
SS(i_*\bC_U) = \lim_{\ep\to 0^+} \Ga_{\ep d\log m}.
$$
where $m:M_\bR\to \bR_{\geq 0}$, 
$m\bigr|_U >0$, $m\bigr|_{\partial U} =0$.
\end{theorem}

\begin{example}\label{interval}
Let $U=(0,1)\subset \bR$. Figure \ref{fig:standLag} depicts the standard Lagrangian on $U$ in $T^*\bR\cong \bR^2$, while Figure \ref{fig:singular} depicts the singular supports of standard and costandard constructible sheaves supported on this interval.
\begin{figure}[h]
\psfrag{1}{\small $m=x(1-x)$}
\psfrag{2}{\small $\log m = \log x + \log(1-x)$}
\psfrag{3}{\small $d\log m = \frac{dx}{x}-\frac{dx}{1-x}.$}
\includegraphics[scale=0.4]{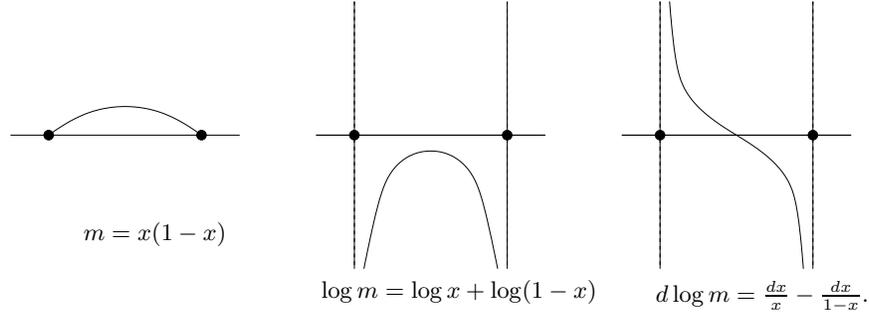}
\caption{The graphs of $m$, $\log m$ and $d\log m$. The graph of $d\log m$ is the standard Lagrangian over $U$.}
\label{fig:standLag}
\end{figure}
\begin{figure}[h]
\psfrag{standard}{$SS(i_*\bC_U)$}
\psfrag{costandard}{$SS(i_!\omega_U)$}
\includegraphics[scale=0.7]{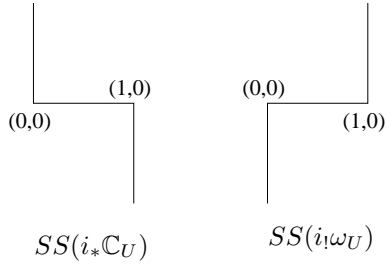}
\caption{Singular supports of standard and costandard sheaves associated to $U=(0,1)$}
\label{fig:singular}
\end{figure}
\end{example}

Given a submanifold $Y\subset X$, let $T_Y^*X$ denote the conormal bundle of $Y$ in $X$. $T_Y^*X$ is a Lagrangian
submanifold of $T^*X$.

\begin{example}[open sets with smooth boundaries]\label{disk}
Let $U$ be an open subset of $\bR^n$, and suppose
that the boundary $\partial U$ is a smooth $n-1$-dimensional submanifold
of $\bR^n$. (This includes Example \ref{disk} as a special case.)

Let $\nu:\partial U\to T^*_{\partial U}\bR^n$ be a nowhere zero
section such that $\nu_x(v_x)>0$ if $v_x$ is an outward normal at $x\in \partial U$.
Then
$$
T^*_U\bR^n= U\times \{0\}\subset T^*\bR^n,\quad
T^*_{\partial U}\bR^n=\{(x,t \nu_x)\mid x\in \partial U, t\in \bR \}.
$$
\begin{eqnarray*}
SS(i_*\bC_U)&=& T^*_U\bR^n\cup \{ (x,t\nu_x)\mid x\in \partial U, t\leq 0\}, \\
SS(i_!\omega_U)&=& T^*_U\bR^n \cup \{ (x,t\nu_x)\mid x\in \partial U, t\geq  0\}.
\end{eqnarray*}
For example, let $D$ be an open disk in $\bR^2$, and identify conormal vectors
with normal vectors. The singular supports of $i_*\bC_D$ and $i_!\omega_D$ are
depicted in Figure \ref{fig:disk} below.
\begin{figure}[h]
\psfrag{standard}{$SS(i_*\bC_D)$}
\psfrag{costandard}{$SS(i_!\omega_D)$}
\includegraphics[scale=0.5]{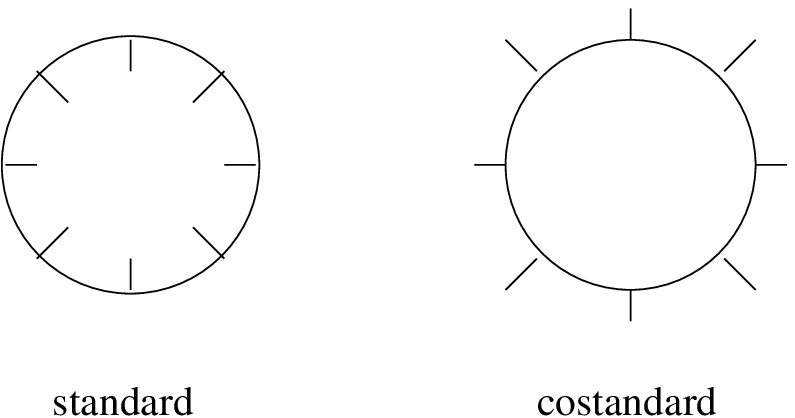}
\caption{Singular supports of standard and costandard sheaves associated to an open disk $D\subset \bR^2$}
\label{fig:disk}
\end{figure}
\end{example}

\begin{example}[manifold with corners]
We can also consider an open set $U$ in $\bR^n$ such that the closure
$\overline{U}$ of $U$ is a manifold with corners. 
For example, an open square $R$ in $\bR^2$. 
\begin{figure}[h]
\psfrag{standard}{$SS(i_*\bC_R)$}
\psfrag{costandard}{$SS(i_!\omega_R)$}
\includegraphics[scale=0.6]{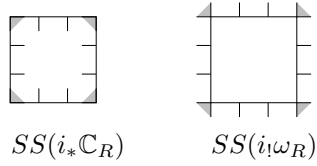}
\caption{Singular supports of standard and costandard sheaves associated to an open square $R\subset \bR^2$ }
\label{fig:square}
\end{figure}
\end{example}


There is a monoidal structure $\star$ on the dg category $Sh_c(X)$ when $X$ is an abelian group, given by the convolution product of constructible sheaves:
$$
\cF\star \cG = a_! (\underbrace{\cF \boxtimes \cG}_{\in Sh(X \times X)}),
$$
where $a: X \times X \to X$ is the addition map of the abelian group $X$. This product turns out to be a tensor product (commutative monoidal).

\subsection{Theta sheaves}
\label{sec:theta}

In \cite{stack} we have introduced the concept of \emph{theta sheaves}, as building blocks of coherent-constructible 
correspondence. Let $\cX_\bSi$ be the toric orbifold defined by a stacky fan $\bSi=(N,\Si,\beta)$.

\subsubsection{Quasicoherent theta sheaves}
The quasicoherent theta sheaves are certain $T$-equivariant quasicoherent
sheaves on $\cX_\bSi$ that arise in the \v{C}ech resolution with respect to an
equivariant open cover of $\cX_\bSi$. We first describe this
open cover.  Given a $d$-dimensional cone $\si\in \Si$, let 
$z_\si = \prod_{\rho_i \not\subset \si} z_i$. Then
$$
U_\si=\{ (z_1,\ldots, z_r)\in \bC^r\mid z_\si\neq 0\}\cong \bC^d\times (\bC^*)^{r-d}
$$ 
is a Zariski open subset of 
$$
U_{\bSi}=\bigcup_{\si\in\Si} U_\si, 
$$
The open embedding $U_\si\hookrightarrow U_\bSi$ descends to an open embedding
of stacks:
$$
j_\si: \cX_\si:=[U_\si/G_\bSi]\hookrightarrow \cX_\bSi= [U_\bSi/G_\bSi].
$$
Then $\{ \cX_\si\mid \si\in \Si\}$ is an open cover of $\cX_\bSi$.

We now describe $T$-equivariant line bundles on $\cX_\si$, or equivalently,
the $\tT$-equivariant line bundles on $U_\si$. We first introduce some notation.
\begin{itemize}
\item Let $M_\bR:=M\otimes \bR$ be the dual vector space of $N_\bR$, and let
$\langle -, - \rangle: M_\bR\times N_\bR\to \bR$ be the natural pairing.
\item Given a $d$-dimensional cone $\si\in \Si$, let $N_\si\subset N$ be the subgroup generated by
$\{ b_i\mid \rho_i\subset \si\}$, and $M_\si$ be the dual lattice
$M_\si=\Hom(N_\si,\bZ)$. Then $N_\si$ and $M_\si$ are free abelian groups of
rank $d$. Let $\langle -, - \rangle_\si: M_\si\times N_\si\to \bZ$
be the natural pairing. 
\end{itemize}
The $T$-equivariant line bundles on $\cX_\si$ are in one-to-one correspondence
with the elements in $M_\si$. Let $\cO_{\cX_\si}(\chi)$ denote
the $T$-equivariant line bundle on $\cX_\si$ associated to $\chi\in M_\si$. 

We define the quasicoherent theta sheaf $\Theta'(\si,\chi)$ to be the
pushforward of $\cO_{\cX_\si}(\chi)$ under the open embedding
$j_\si:\cX_\si\hookrightarrow \cX_\bSi$.
$$
\Theta'(\si,\chi):=j_{\si*} \cO_{\cX_\si}(\chi).
$$

\subsubsection{Constructible theta sheaves}
For a cone $\si\in \Si$ and a character $\chi\in M_\si$, we fix the following notation:
\begin{eqnarray*}
\si^\vee_\chi &=& \{ x\in M_\bR \ \langle x, v \rangle \ge \langle \chi, v \rangle_\si,\ v\in N_\si\cap \si \},\\
(\si^\vee_\chi)^\circ &=& \{x\in M_\bR| \langle x, v \rangle> \langle \chi, v \rangle_\si,\ v\in N_\si\cap \si \},\\
\si^\perp_\chi&=&\{x\in M_\bR| \langle x, v \rangle= \langle \chi, v \rangle_\si,\ v\in N_\si\cap \si \}.
\end{eqnarray*}

We define the constructible theta sheaf $\Theta(\si,\chi)$ to be the costandard constructible
sheaf associated to the open set $(\si^\vee_\chi)^\circ$ in $M_\bR$.
$$
\Theta(\si,\chi):= i_{(\si_\chi^\vee)^\circ !} \omega_{ (\si_\chi^\vee)^\circ} \in Ob(Sh_c(M_\bR)),
$$
where $i_{(\si_\chi^\vee)^\circ}: (\si^\vee_\chi)^\circ \hookrightarrow M_\bR$ is the inclusion.

\subsection{The coherent-constructible correspondence}
The theta sheaves are indexed by the set 
$$
\Ga(\bSi)=\{(\si,\chi)|\si\in \Si, \chi\in M_\si\}.
$$
We define a partial order on $\Ga(\bSi)$:
$$
(\si_1,\chi_1)\leq (\si_2,\chi_2) \textup{ if and only if }(\si_1)^\vee_{\chi_1}\subset (\si_2)^\vee_{\chi_2}.
$$
The ``linearized" dg category $\Gamma(\bSi)_\bC$ consists of objects $(\si,\chi)\in\Gamma(\bSi)$ and the following morphisms with obvious composition rules 
$$
hom((\si_1,\chi_1),(\si_2,\chi_2))=
\begin{cases}\bC[0]\text{, if $(\si_1,\chi_1)\leq (\si_2,\chi_2)$;}\\ 
0\text{, otherwise.}\end{cases}
$$ 
It is proved in \cite{stack} that
\begin{equation}\label{eqn:Theta-Theta-prime}
\begin{aligned}
 &hom_{\cQ_T(\cX_\bSi)}(\Theta'(\si_1,\chi_1),\Theta'(\si_2,\chi_2))\\
=&hom_{Sh_c(M_\bR)}(\Theta(\si_1,\chi_1),\Theta(\si_2,\chi_2))\\
=&hom((\si_1,\chi_1),(\si_2,\chi_2))
\end{aligned}
\end{equation}
for any $(\si_1,\chi_1),(\si_2,\chi_2)\in \Ga(\bSi)$.

Let $\ltr_\bSi$ (resp. $\ltrp_{\bSi}$) be the full triangulated
subcategory of $Sh_c(M_\bR)$ (resp. $\cQ_T(\cX_\bSi)$)  generated by all $\Theta(\si,\chi)$ (resp. $\Theta'(\si,\chi)$).
Then \eqref{eqn:Theta-Theta-prime} implies that
$\ltr_\bSi$ and $\ltrp_\bSi$ are quasi-equivalent as triangulated 
dg categories.  We proved that this quasi-equivalence is monoidal:
\begin{theorem} \label{thm:ccc_theta}
There is a quasi-equivalence monoidal functor $\kappa_\bSi: \ltrp_\bSi\to\ltr_\bSi$, 
which sends $\Theta'(\si,\chi)$ to $\Theta(\si,\chi)$ for $(\si, \chi)\in \Ga(\bSi)$.
\end{theorem}

By \v{C}ech resolution, the dg category of coherent sheaves $\Perf_T(\cX_\bSi)$ is a full subcategory of $\ltrp_\bSi$, as shown in \cite{stack}. 
Restricted to $\Perf_T(\cX_\bSi)$, the functor $\kappa_\bSi$ is a quasi-embedding 
(full and faithful at the cohomology level). We have characterized the image $\kappa_\bSi(\Perf_T(\cX_\bSi))$ as 
a full sub-category of $\ltr_\bSi$, as in the following theorem.
\begin{theorem}[coherent-constructible correspondence for toric orbifolds]
\label{thm:ccc}
Let $\cX_\bSi$ be a complete toric orbifold defined
by a stacky fan $\bSi=(N,\Si,\beta)$. Then there is
a quasi-equivalence of monoidal triangulated dg categories: 
$$
\kappa_\bSi: \Perf_T(\cX_\bSi)\stackrel{\sim}{\lra} \Sh_{cc}(M_\bR,\LbS). 
$$
\end{theorem}
In the above theorem, the dg category $Sh_{cc}(M_\bR,\LbS)$ is the full dg subcategory of $Sh_{cc}(M_\bR)$
on $\cX_\bSi$ whose objects have singular support inside $\LbS$. It is closed under the monoidal product $\star$. The conical Lagrangian ($\bR_{>0}$-invariant Lagrangian in $T^*M_\bR$) is defined directly from the stacky fan $\bSi=(N,\Si,\beta)$. 
$$
\LbS=\bigcup_{\si\in \Si, \chi\in M_\si} \si^\perp_\chi\times (-\si) \subset 
M_\bR\times N_\bR =T^*M_\bR.
$$
By definition $\LbS$ is a conical Lagrangian in $T^*M_\bR$, and is invariant
under $(x,y)\mapsto (x+m,y)$, $m\in M$. 

\begin{remark}
It is particularly easy to describe what the functor $\kappa_\bSi$ does
to $\bQ$-ample equivariant line bundles.
Recall that any $T$-equivariant line bundle on $\cX=\cX_\bSi$ is of the form 
$\cL_\vc=\cO_{\cX}(c_1\cD_1+\cdots c_r\cD_r)$, where
$\cD_i$ denotes the $T$-divisor associated to
the ray $\rho_i\in \Si(1)$, and $c_1,\ldots, c_r\in \bZ$ are integers.
We say $\cL_\vc$ is $\bQ$-ample if there is some positive integer $n$ such
that $\cL_\vc^{\otimes n}$ is the pull back of an ample line bundle
on the coarse moduli space. If $\cL_\vc$ is $\bQ$-ample then 
$$
\tri_\vc=\{ x\in M_\bR\mid \langle x, b_i\rangle \geq -c_i\}
$$
is a convex polytope in $M_\bR$. The interior $\tri_\vc^\circ$
of $\tri_\vc$ is a bounded open set in $M_\bR$. Let
$i:\tri_\vc^\circ\hookrightarrow M_\bR$ be the inclusion. Then
$$
\kappa_{\bSi}(\cL_\vc) = i_! \omega_{\tri^\circ_\vc}.
$$
\end{remark}

We have also proved a coherent-constructible correspondence
for the coarse moduli space $X_\Si$ \cite{more}. Indeed, we prove the following
for any complete (not necessarily simplicial) toric varieties:

\begin{theorem}[coherent-constructible-correspondence for toric varieties]
Let $X_\Si$ be a complete toric variety  defined
by a fan $\Si\subset N_\bR$. Then there is
a quasi-equivalence ot monoidal triangulated dg categories: 
$$
\kappa_\bSi: \Perf_T(X_\Si)\stackrel{\sim}{\lra} \Sh_{cc}(M_\bR,\LS). 
$$
\end{theorem}
The category $Sh_{cc}(M_\bR;\LS)$ is similarly defined as the subcategory of 
$Sh_{cc}(M_\bR)$ whose objects have singular support in
a conical Lagrangian
$$
\LS:= \bigcup_{\si\in \Si,\chi\in M}
(\si^\perp+\chi) \times (-\si) \subset M_\bR\times N_\bR= T^*M_\bR,
$$
where $\si^\perp =\{ x\in M_\bR\mid \langle x,v\rangle =0\textup{ for all } v\in \si\}$.

The theorem above can be considered as a ``categorification" of Morelli's theorem \cite{Mo}.
Let $L_M(M_\bR)$ be the group of functions generated over $\bZ$ by the indicator functions
$1_P$ of convex lattice polyhedra $P$, and let $S_M(M_\bR)$ be the 
abelian group generated by rational convex cones in $M_\bR$. 
Then $S_M(M_\bR)$ is the  group of germs of functions in $L_M(M_\bR)$ at the origin
(or any point in the lattice $M$). Let $S_\Sigma(M_\bR)$ be the subgroup
of $S_M(M_\bR)$ generated by $\{\si^\vee\mid \si\in \Si\}$.

\begin{theorem}[Morelli]
Let $X_\Si$ be a smooth projective toric variety. Then
there is a group isomorphism $K_T(X_\Si)\stackrel{\sim}{\lra}S_\Si(M_\bR)$,
$\cL_\vc\mapsto 1_{\tri_\vc}$, $\cL_\vc$ ample.
\end{theorem}

Bondal has also proved a similar relation between (non-equivariant) coherent sheaves and constructible sheaves \cite{Bo}, characterized by a stratification.
\begin{theorem}[Bondal]
Let $X_\Si$ be a smooth projective variety defined by
a fan $\Si$,  with some additional
assumption on $\Si$.
$$
D^b Coh(X_\Si)\cong D\Sh_c(M_\bR/M, \cS)
$$
where $\cS$ is a stratification on the compact
torus $M_\bR/M\cong (S^1)^n$ determined
by $\Si$.
\end{theorem}

\begin{example}[Example \ref{ex:stack2} continued]
Let the stacky fan $\bSi$ be as in Example \ref{ex:stack2}, which defines the toric orbifold $\bP(1,3)$. 
The conical Lagrangians $\LS$ and $\LbS$ are shown in Figure \ref{fig:LS}.
\begin{figure}[h]
\psfrag{LS}{$\LS$}
\psfrag{LbS}{$\LbS$}
\psfrag{-1}{\Tiny $-1$}
\psfrag{0}{\Tiny $0$}
\psfrag{1}{\Tiny $1$}
\psfrag{-1/3}{\Tiny $-1/3$}
\psfrag{-2/3}{\Tiny $-2/3$}
\psfrag{1/3}{\Tiny $1/3$}
\psfrag{2/3}{\Tiny $2/3$}
\includegraphics[scale=0.3]{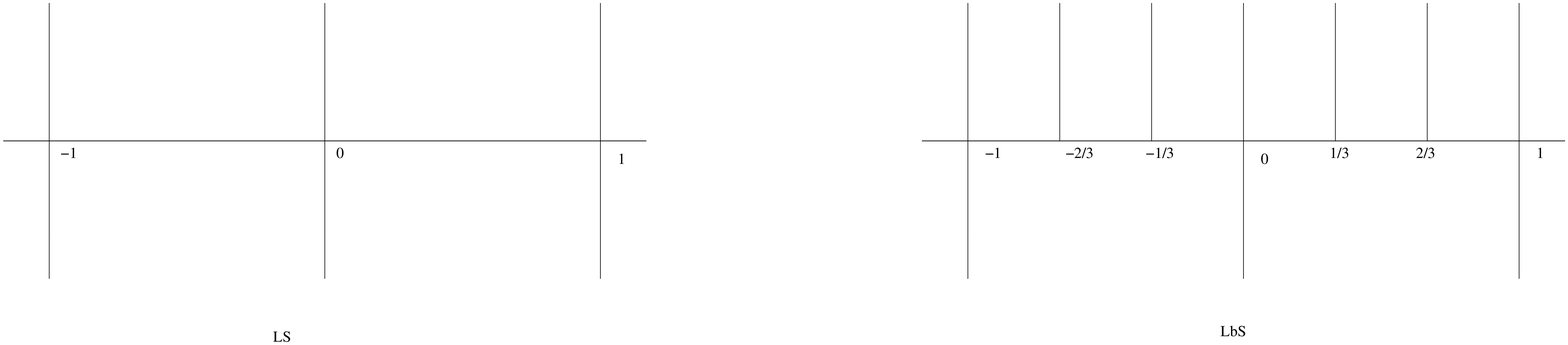}
\caption{The conical Lagrangians $\LS$ and $\LbS$ for $\bSi$ defined in Example \ref{ex:fan2}. The horizontal direction is $M_\bR$ and the vertical direction is $N_\bR$.}
\label{fig:LS}
\end{figure}
\end{example}

\begin{example}[Example \ref{ex:stack1} continued]
\begin{enumerate}
\item $T$-equivariant ample line bundle on $\bP^1$:
$\cO(c_1 D_1+ c_2 D_2)$, $c_1, c_2\in \bZ$, $c_1+c_2>0$.
$$
\tri_{c_1,c_2}^\circ =\{ x\in \bR\mid x > -c_1, -x > -c_2\} =(-c_1, c_2)
$$

\item $T$-equivariant $\bQ$-ample line bundles on $\bP(1,3)$: $\cO(c_1\cD_1 + c_2\cD_2)$, $c_1, c_2\in \bZ$, $\frac{c_1}{3} + c_2 >0$. Let $p:\bP(1,3)\to \bP^1$ be the projection to the coarse moduli space. Then $p^*D_1 = 3\cD_1$, $p^*D_2=\cD_2$.
$$
\tri_{c_1,c_2}^\circ=\{ x\in \bR \mid 3x>-c_1, -x>c_2\}=(-\frac{c_1}{3}, c_2).
$$ 
\end{enumerate}
\end{example}

\section{Fourier-Mukai Transformation: A Constructible Perspective}
\label{sec:FM}

In this section, $X_1$ and $X_2$ are always simplicial toric varieties
defined by simplicial fans $\Si_1$ and $\Si_2$ in $N_\bR$, respectively; $B$ and $C$ are effective toric 
$\bQ$-divisors on $X_1$ and $X_2$, respectively, such that
$(X_1,B)$ and $(X_2,C)$ are toric log pairs as in Section \ref{sec:log-pair}. 
Let $\cX_1$ and $\cX_2$ be the toric orbifolds associated to the
pairs $(X_1,B)$ and $(X_2,C)$, respectively.  
Assume there are proper birational morphisms $\mu_1: W\to X_1$ and $\mu_2: W\to X_2$ for some variety $W$ 
such that $\mu_1^*(K_{X_1}+B)\ge \mu_2*(K_{X_2}+C)$. Kawamata conjectures that there exists a full and faithful functor of triangulated 
categories
$$
F'_{12}=\mu_{1*}\circ  \mu_2^*: D^b\Coh(\cX_2)\to D^b\Coh(\cX_1)
$$
in \cite{Ka}, where $\mu_i:\cW\to \cX_i$ are the morphisms for the corresponding
stacks, by an abuse of notation. If the inequality above becomes an equality, by invoking this conjecture in both directions, $F'_{12}$ is then an equivalence of triangulated categories. Passing Kawamata's argument to the language of constructible sheaves via Theorem \ref{thm:ccc}, 
this Fourier-Mukai fully faithful functor arises from intuitive combinatorial argument. 
This section elaborates Kawamata's theorem from the perspective of constructible sheaves, proving some cases discussed in \cite{Ka}, in the equivariant and dg setting.\footnote{Although not explicitly stated, Kawamata's proof is essentially equivariant in \cite{Ka}.}

We introduce some notation:
\begin{enumerate} 
\item The Fourier Mukai functors are
$F'_{12}=\mu_{1*}\circ \mu_2^*$ and
$F'_{21}=\mu_{2*}\circ \mu_1^*$.
\item Let $D_{1,i}$ (resp. $\cD_{1,i}$) denote the $T$-divisor on 
$X_1$ (resp. $\cX_1$) associated to the 1-dimensional cone
$\rho_i\in \Si_1(1)$.

\item Let $D_{2,i}$ (resp. $\cD_{2,i}$) denote the $T$-divisor on 
$X_2$ (resp. $\cX_2$) associated to the 1-dimensional cone
$\rho_i\in \Si_2(1)$.

\item Let $D'_i$ (resp. $\cD'_i$) denote the $T$-divisor
on $W$ (resp. $\cW$) associated to the 1-dimensional cone
$\rho_i\in \Si'(1)$.

\item Let $p_i:\cX_i\to X_i$ be canonical map to the coarse moduli. 
\end{enumerate}

We have
$$
B=\sum_{i=1}^{l_1} (1-\frac{1}{r_i}) D_{1,i},\quad C=\sum_{i=1}^{l_2} (1-\frac{1}{s_i}) D_{2,i},
$$
where $r_i$, $s_i$ are positive integers. Then
$$
p_1^*D_{1,i}=r_i \cD_{1,i},\quad p_2^*D_{2,i}=s_i \cD_{2,i}.
$$
Note that from the construction of $\cX_i$, 
$$
p_1^*(K_{X_1}+B)=K_{\cX_1},\quad p_2^*(K_{X_2}+B)=K_{\cX_2}.
$$

\subsection{Toric orbifolds with the same coarse moduli space}\label{sec:caseI}

In the first case of \cite[Theorem 4.2]{Ka}, Kawamata shows that if $X_1=X_2=X$ and $K_{X_1}+B\ge K_{X_2}+C$, 
then the Fourier-Mukai functor
$$
F'_{12}=\mu_{1*}\circ \mu_2^*: \Coh(\cX_2)\to \Coh(\cX_1)
$$
is fully faithful.


Recall from Section \ref{sec:definition} that $N=\bZ^n$, and $\Si$ is a simplicial fan in $N_\bR$. The $1$-cones $\Si(1)$ consists of rays $\rho_1,\dots,\rho_l$, and the generating set of $\Si(1)\cap N$ is $\{v_1,\dots, v_r\}$. Let $\beta_1$ and $\beta_2$ to be maps (where $v_i$ are regarded as column vectors below)
\begin{eqnarray*}
&& \beta_1=\left[\begin{array}{ccc} r_1 v_1 & \cdots & r_l v_l \end{array}\right]:
\bZ^l\lra N=\bZ^n,\\
&& \beta_2=\left[\begin{array}{ccc} s_1 v_1 & \cdots & s_l v_l \end{array}\right]:
\bZ^l\lra N=\bZ^n.
\end{eqnarray*}
>From the stacky fans $\bSi_i=(N,\Si,\beta_i)$ one defines two toric DM stacks 
$\cX_1=\cX_{\bSi_1}$ and $\cX_2=\cX_{\bSi_2}$. They have the same coarse moduli space 
$X=X_\Si$ given by the fan $\Si$ as a toric variety.

Let $\cW=\cX_1\times_X\cX_2$. It is the toric orbifold defined by the stacky fan $\bSi'=(N,\Si,\beta')$, where
$$
\beta'=\left[\begin{array}{ccc} t_1 v_1 & \cdots & t_l v_l \end{array}\right]:
\bZ^l\lra N=\bZ^n,\quad t_i=l.c.m.(r_i,s_i).
$$
We have the following diagram
$$
\xymatrix{
& \cW \ar[dl]_{\mu_1}  \ar[dr]^{\mu_2} &  \\
\cX_1 \ar[dr]^{p_1} &  &\cX_2 \ar[dl]_{p_2} \\
& X  &}
$$
where $p_i$ is the morphism from $\cX_i$ to their common coarse moduli space $X$.
Given a 1-dimensional cone $\rho_i\in \Si(1)$ let $D_i$, $\cD_{1,i}$, $\cD_{2,i}$, and $\cD'_i$ denote the
associated $T$-divisors on $X$, $\cX_1$, $\cX_2$, and $\cW$, respectively. 
Let $m_i=\frac{t_i}{r_i} \in \bZ$ and $n_i=\frac{t_i}{s_i}\in \bZ$.

$$
p_1^*D_i = r_i \cD_{1,i},\quad p_2^*D_i = s_i \cD_{2,i}, \quad
\mu_1^*\cD_{1,i}= m_i \cD'_i,\quad \mu_2^* \cD_{2,i}= n_i \cD'_i.
$$ 
\begin{eqnarray*}
&& K_X = - \sum_{i=1}^l D_i,\quad B = \sum_{i=1}^l(1-\frac{1}{r_i}) D_i,\quad
C=\sum_{i=1}^l(1-\frac{1}{s_i}) D_i\\
&& p_1^*(K_X+B)=-\sum_{i=1}^l \cD_{1,i}= K_{\cX_1},\quad
p_2^*(K_X+C)= -\sum_{i=1}^l \cD_{2,i}= K_{\cX_2}\\
&& \mu_1^*K_{\cX_1} = -\sum_{i=1}^l m_i \cD_i',\quad
\mu_2^*K_{\cX_2} = -\sum_{i=1}^l n_i \cD_i'.
\end{eqnarray*}

$ $From the above calculations, we observe that: 
\begin{lemma} \label{lem:K-I}
$$
\mu_1^*K_{\cX_1}\geq \mu_2^* K_{\cX_2} \Leftrightarrow  r_i\geq s_i, \  i=1,\ldots,l
\Leftrightarrow  K_X+ B\geq K_X +C.
$$
\end{lemma}

For any $\si\in \Si(d)$, let $\{ v_{i_1},\ldots, v_{i_d}\}= \si\cap \{ v_1,\ldots, v_l\}$.
There are an injective group homomorphisms
\begin{eqnarray*}
&& \mu_{1,\si}:N'_\si=\bigoplus_{k=1}^d \bZ(t_{i_k} v_{i_k})\lra  
      N_{1,\si}=\bigoplus_{k=1}^d \bZ(r_{i_k} v_{i_k})\\
&& \mu_{2,\si}:N'_\si=\bigoplus_{k=1}^d \bZ(t_{i_k} v_{i_k})\lra  
      N_{2,\si}=\bigoplus_{k=1}^d \bZ(s_{i_k} v_{i_k})
\end{eqnarray*}
and surjective group homomorphisms 
$$
\mu_{i,\si}^*: M_{i,\si}:= \Hom(N_{i,\si},\bZ)\to  M'_\si:=\Hom(N'_\si,\bZ),\quad i=1,2.
$$ 

We now introduce some notation. Given $\si\in \Si$, let
$\langle \ , \ \rangle'_\si : M_\si'\times N'_\si\to\bZ$ and
$\langle \ , \ \rangle_{i,\si}: M_{i,\si}\times N_{i,\si}\to \bZ$,
$i=1,2$, be the natural pairing. Given $x\in \bR$, define
$\lceil x \rceil \in \bZ$ by $\lceil x \rceil -1< x\leq \lceil x\rceil$.
We define surjective maps (which is not a group homomorphism) 
$\mu_{i,\si*}: M'_{\si}\to M_{i,\si}$, $i=1,2$, by 
\begin{eqnarray*}
&& \langle \mu_{1,\si*}(\chi),(r_{i_k} v_{i_k}) \rangle_{1,\si} = 
\lceil \frac{1}{m_{i_k}}\langle \chi,(t_{i_k}v_{i_k})\rangle'_\si\rceil \in \bZ \\
&& \langle \mu_{2,\si*}(\chi),(s_{i_k} v_{i_k}) \rangle_{2,\si} = 
\lceil \frac{1}{n_{i_k}}\langle \chi,(t_{i_k}v_{i_k})\rangle'_\si\rceil \in \bZ \\ 
\end{eqnarray*}
where
$$
\chi\in M'_\si,\quad
\frac{1}{m_{i_k}}\langle \chi_i, (t_{i_k} v_{i_k})\rangle_\si \in  \frac{1}{m_{i_k}}\bZ,
\quad  \frac{1}{n_{i_k}}\langle \chi_i, (t_{i_k} v_{i_k})\rangle_\si 
\in  \frac{1}{n_{i_k}}\bZ,
$$

For $i=1,2$, define (with an abuse of notation)
\begin{eqnarray*}
&& \mu_i^*:\Gamma(\bSi_i)\to \Gamma(\bSi'),\quad (\si,\chi)\mapsto (\si,\mu_{i,\si}^*\chi)\\
&& \mu_{i*}:\Gamma(\bSi')\to \Gamma(\bSi_i),\quad (\si,\chi)\mapsto (\si,\mu_{i,\si *} \chi)
\end{eqnarray*}

Let $\Theta_1(\si,\chi)$ (resp. $\Theta_\cW(\si,\chi)$, $\Theta_2(\si,\chi)$) be the constructible theta sheaves on $M_{1,\bR}$ (resp. $M'_\bR$, $M_{2,\bR}$) for $\si\in \Si_1$ (resp. $\Si'$, $\Si_2$) and $\chi \in M_{1,\si}$ (resp. $M'_\si$, $M_{2,\si}$). Similarly, let $\Theta'_1(\si,\chi)$ (resp. $\Theta'_\cW(\si,\chi)$, $\Theta'_2(\si,\chi)$) be the quasi-coherent theta sheaves on $\cX_1$ (resp. $\cX_\cW$, $\cX_2$) for $\si\in \Si_1$ (resp. $\Si'$, $\Si_2$) and $\chi \in M_{1,\si}$ (resp. $M'_\si$, $M_{2,\si}$).
\begin{proposition} \label{pro:pullback-pushforward-I}
For $i=1,2$, let $\mu_i^*:\cQ_T(\cX_i)\to \cQ_T(\cW)$ and $\mu_{i*}:\cQ_T(\cW)\to \cQ_T(\cX_i)$
be the pullback and pushforward functors of equivariant quasicoherent sheaves. Then
\begin{eqnarray*}
&& \mu_i^*\Theta_i'(\si,\chi) =\Theta_\cW'(\mu_i^*(\si,\chi)),\quad (\si,\chi)\in \Ga(\bSi_i),\\
&& \mu_{i*}\Theta_\cW'(\si,\chi) =\Theta_i'(\mu_{i*}(\si,\chi)),\quad (\si,\chi)\in \Ga(\bSi').
\end{eqnarray*}
\end{proposition}
\begin{proof}
It suffices to consider the case $i=1$.
The first statement follows directly from the functoriality property of CCC 
\cite[Theorem 5.16]{stack}.
For the second statement, the theta sheaf $\Theta_\cW'(\si,\chi)$ is given by the module 
$\bC[\si^\vee_\chi \cap M'_\si]$. 
The sections of the push-forward are the sections of $\Theta_\cW'(\si,\chi)$ 
whose characters are in $M_{1,\si}$. Thus $\mu_{1*}\Theta_\cW'(\si,\chi)$ is given by the 
module $\bC[\si^\vee_\chi \cap M_{1,\si}]$. Notice that 
$\si^\vee_\chi \cap M_{1,\si}=\si^\vee_{\mu_{1*} (\chi)} \cap M_{1,\si}$, and the result follows.
\end{proof}

\begin{proposition}\label{pro:F-I} 
If $r_i\geq  s_i$ for $i=1,\ldots,l$. Then 
$$
F:=\mu_{1*}\circ \mu_2^* :\Gamma(\bSi_2)\to \Gamma(\bSi_1)
$$ 
is an injective map of posets:
$$
(\si,\chi)\leq (\si',\chi')\Leftrightarrow
F(\si,\chi) \leq F(\si',\chi').
$$
\end{proposition}
\begin{proof} For any $\si\in \Si$, let $F_\si= \mu_{1,\si *}\mu^*_{2,\si}: M_{2,\si}\to M_{1,\si}$.
By definition, 
$$
F(\si,\chi)=\left(\si, F_\si (\chi) \right),\quad 
F(\si',\chi')=\left(\si', F_{\si'}(\chi')\right).
$$
The statements (i) and (ii) below  also follow from the definitions.
\begin{enumerate}
\item[(i)] Suppose that $(\si,\chi), (\si,\chi')\in \Ga(\bSi_2)$. Then
\begin{eqnarray*}
&& (\si,\chi)\leq (\si',\chi')\\
&\Leftrightarrow& \si\supset \si' \textup{ and } \langle \chi, s_i v_i \rangle_{2,\si}
\geq \langle \chi',s_i v_i \rangle_{2,\si} \textup{ for all } v_i\in \si'\cap\{ v_1,\ldots, v_l\}.
\end{eqnarray*}
\begin{eqnarray*}
&& (\si,F_\si(\chi))\leq (\si',F_\si(\chi)) \\
&\Leftrightarrow&   \si\supset \si' \textup{ and }
\langle F_\si(\chi), r_i v_i \rangle_{1,\si} \geq \langle F_{\si'}(\chi_i'), r_i v_i\rangle_{1,\si} 
\textup{ for all } v_i \in \si'\cap \{ v_1,\ldots, v_l\}.
\end{eqnarray*}
\item[(ii)] If $(\chi,\si)\in M_{2,\si}$ and $v_i\in \si\cap \{ v_1,\ldots, v_l\}$ then 
$$
\langle F_\si(\chi), r_i v_i\rangle_{1,\si}=\lceil \frac{r_i}{s_i}\langle \chi, s_i v_i \rangle_{2,\si}\rceil.
$$
\end{enumerate}
By (i) and (ii), it suffices to show that for any $k,k'\in \bZ$,  
$$
k\ge k' \Longleftrightarrow \lceil \frac{r_i}{s_i}k \rceil \ge \lceil \frac{r_i}{s_i}k'\rceil.
$$
$\Rightarrow$ is always true, and $\Leftarrow$ is true if $\frac{r_i}{s_i}\geq 1$.
\end{proof}

\begin{theorem}
Suppose that $\mu_1^*K_{\cX_1}\geq \mu_2^*K_{\cX_2}$ (or equivalently, $r_i\geq s_i$ for $i=1,\ldots,l$).
Then the dg functor 
$$
F_{12}':\Coh_T(\cX_2)\to \Coh_T(\cX_1)
$$ 
is cohomologically full and faithful.
\end{theorem}
\begin{proof}
The functor $F_{12}':\cQ_T(\cX_2)\to \cQ_T(\cX_1)$, restricted on $\langle \Theta'_2 \rangle \subset \cQ_T(\cX_2)$, is a functor to $\langle \Theta'_1 \rangle$, since it sends any theta sheaf to a theta sheaf. Therefore $F_{12}'$ restricted on $\langle \Theta'_2 \rangle \subset \cQ_T(\cX_2)$ is a full and faithful functor since $F_{12}:=\kappa_1 \circ F_{12}' \circ \kappa_2^{-1}$ is full and faithful due to Lemma \ref{lem:K-I}, Proposition \ref{pro:pullback-pushforward-I}, and Proposition \ref{pro:F-I}. Further restricting this functor to coherent sheaves (i.e. perfect sheaves), we obtain a full and faithful functor $F_{12}':\Coh_T(\cX_2)\to \Coh_T(\cX_1)$.
\end{proof}

\begin{example}
Set $N=\bZ$, $\Si=\{\bR^+,\bR^-\}$, $\beta_1=[3\ -1]$ and $\beta_2=[2\ -1]$. The stacks associated to $\bSi_1=(N,\Si, \beta_1)$ and $\bSi_2=(N,\Si,\beta_2)$ are
$$\cX_1=\bP(1,3),\quad \cX_2=\bP(1,2).$$
Both $\cX_1$ and $\cX_2$ have the same coarse moduli space $\bP^1$. The fiber product $\cW=\bP(1,6)$ is constructed from the stacky 
fan $(N,\Si, [6\  -1])$. Let $\rho_1=\bR^+$ and $\rho_2=\bR^-$. 

For $i=1,2$, let $D_i\subset X=\bP^1$, $\cD_{1,i}\subset \cX_1$, $\cD_{2,i}\subset \cX_2$, and
$D'_i\subset \cW$ be defined as above. In particular, $D_1$ and $D_2$ are the two torus fixed points
in $\bP^1$. Any equivariant line bundle on $\cX_1= \bP(1,3)$ is of the form
$$
\cL_{1, (c_1,c_2)}:= \cO_{\bP(1,3)}(c_1\cD_{1,1}+c_2\cD_{1,2}) = p_1^*\cO_{\bP^1}(\frac{c_1}{3} D_1 + c_2 D_2),\quad c_1,c_2\in \bZ,
$$
whereas any equivariant line bundle on $\cX_2=\bP(1,2)$ is of the form
$$
\cL_{2, (c_1,c_2)}:= \cO_{\bP(1,2)}(c_1\cD_{2,1}+c_2\cD_{2,2}) = p_2^*\cO_{\bP^1}(\frac{c_1}{2} D_1 + c_2 D_2),\quad c_1,c_2\in \bZ.
$$

The Fourier-Mukai functor $F=\mu_{1*}\circ \mu_2^*$ is given by 
$$
F_{12}': \cL_{2,(c_1,c_2)} \mapsto   \cL_{1,(\lfloor \frac{3}{2}c_1\rfloor, c_2)}.
$$
The line bundle $\cL_{2,(c_1,c_2)}$ is $\bQ$-ample iff $\frac{c_1}{2}+ c_2 >0$. In this
case, it corresponds to the costandard sheaf supported on the open interval $(-\frac{c_1}{2},c_2)\subset \bR$. 
The constructible analogue of the Fourier-Mukai functor is
$$
F_{12}: i_{(-\frac{c_1}{2},c_2)!} \bC_{(-\frac{c_1}{2},c_2)}[1] \mapsto 
\begin{cases}
i_{(-\frac{c_1}{2},c_2)!} \omega_{(-\frac{c_1}{2},c_2)} & \textup{if $c_1$ is even},\\
i_{(-\frac{c_1}{2}+\frac{1}{6}, c_2)!}  \omega_{(-\frac{c_1}{2} +\frac{1}{6}, c_2)} & \textup{if $c_1$ is odd}.
\end{cases}
$$
where $i_U: U\hookrightarrow \bR$ is the embedding of the corresponding open subset.
Note that $F_{12}'(\cL_{2,(c_1,c_2)})$ is also $\bQ$-ample.
\end{example}

\subsection{Divisorial contraction: overview} \label{sec:contraction}

Let $N=\bZ^n$ and $\si_{X_2}\subset N_\bR$ be a simplicial cone generated by rays $\rho_1,\dots, \rho_n$. 
Let $v_i$ be the primitive generator of $\rho_i\cap N$, and $v_{n+1}=a_1 v_1+\dots+a_{n'}v_{n'}$ for some $n'\leq n$ 
with all $a_i\in \bQ_{>0}$ such that $v_{n+1}\in N$ is primitive. 
Define $\si_{X_1,i_0}$ be the $n$-dimensional cone generated by $v_i$, $1\leq i \le n+1$ with $i\neq i_0$.
Then
$$
\si_{X_2} =\bigcup_{i_0=1}^{n+1} \si_{X_1,i_0}.
$$
Let $\Si_2$ be the fan consisting of the top dimensional cone $\si_{X_2}$ and
its faces, and let $\Si_1$ be the fan consisting of top dimension cones
$\si_{X_1,i_0}$, $1\leq i_0\leq n+1$, and their faces.
Then there is a morphism of fans $\Si_1\to \Si_2$, which induces
a toric morphism $f: X_1=X_{\Si_1}\to X_2=X_{\Si_2}$. Note that
$X_2$ is an affine simplicial toric variety, and $f$ is a toric
divisorial contraction.

Define
\begin{eqnarray*}
&& \beta_1=\left[\begin{array}{cccc} r_1 v_1 & \cdots & r_n v_n & r_{n+1} v_{n+1} \end{array}\right]:
\bZ^{n+1}\lra N=\bZ^n,\\
&& \beta_2=\left[\begin{array}{ccc} r_1 v_1 & \cdots & r_n v_n \end{array}\right]:
\bZ^n\lra N=\bZ^n.
\end{eqnarray*}
For $i=1,2$, one associates a toric orbifold $\cX_i$ to the stacky fan 
$\bSi_i=(N, \Si_i, \beta_i)$.
Let $\rho_{n+1}$ be the ray $\bR^+\cdot v_{n+1}$, and denote 
$r_{n+1}'v_{n+1}$ to be generator of $\bZ \cdot r_{n+1}v_{n+1} \cap N_{2,\si}$. 
Then
$$
r_{n+1}' v_{n+1} =\sum_{i=1}^{n'} a_i' (r_i v_i),\quad 
a_i':= \frac{ r_{n+1}' a_i}{r_i} \in \bZ.
$$
Setting $b_i=r_i v_i$ for $i=1,\dots, n+1$ as in Section \ref{sec:definition}, $\alpha_i=\frac{r_{n+1}}{r_i}a_i\in \bQ_{>0}$ for $i=1,\dots, n'$ and $\alpha_i=0$ for $i=n'+1,\dots, n$, we have $b_{n+1}=\alpha_1 b_1+\dots+\alpha_{n'}b_{n'}$. Let $b'_{n+1}=r_{n+1}'v_{n+1}=m b_{n+1}$, there is a similar relation $b'_{n+1}=\beta_1 b_1+\dots+ \beta_n'b_n'$, where $\beta_i'=\alpha_i'\in \bZ_{>0}$.

Let $\cW$ be the toric orbifold given by the stacky fan 
$$
\bSi'=(N,\Si_1,\beta'=\left[\begin{array}{cccc} r_1 v_1 & \cdots & r_n v_n & r'_{n+1} v_{n+1} \end{array}\right]).
$$
The identity map $N\to N$ defines morphisms of stacky fans $\bSi'\to \bSi_i$, $i=1,2$, which
induce morphisms $\mu_i:\cW\to \cX_i$ of toric orbifolds.
For $j=1,\ldots,n$, let $\cD_{1,j}$, $\cD_{2,j}$, and $\cD'_j$ be $T$-divisors
associated to $\rho_j$ in $\cX_1$, $\cX_2$, and $\cW$, respectively;
let $\cD_{1, n+1}$ and $\cD'_{n+1}$ be the $T$-divisors 
associated to $\rho_{n+1}$ in $\cX_1$ and $\cW$, respectively. Then
for $i=1,\ldots,n$ we have
$$
\mu_1^*\cD_{1,i} =\cD_i',\quad \mu_2^*\cD_{2,i}= \cD_i'+ a_i' \cD'_{n+1},
$$
where $a_i'=0$ for $n'< i\leq n$.
We also have 
$$
\mu_1^*\cD_{1,n+1} = \frac{r_{n+1}'}{r_{n+1}}\cD'_{n+1}.
$$
\begin{eqnarray*}
&& K_{\cX_1} = -\sum_{i=1}^{n+1} \cD_{1,i},\quad 
\mu_1^* K_{\cX_1}= -\sum_{i=1}^n \cD'_i -\frac{r_{n+1}'}{r_{n+1}} \cD_{n+1}'\\
&& K_{\cX_2} = -\sum_{i=1}^n\cD_{2,i},\quad 
\mu_2^* K_{\cX_2} = -\sum_{i=1}^n \cD'_i -\Bigl(\sum_{i=1}^{n'} a_i' \Bigr)\cD_{n+1}'
\end{eqnarray*}
\begin{lemma}
\begin{enumerate}
\item[(a)] $\mu_1^*K_{\cX_1} \geq \mu_2^*K_{\cX_2} \Leftrightarrow  
\displaystyle{ \sum_{i=1}^{n'} \frac{a_i}{r_i} \geq \frac{1}{r_{n+1}} }$.
\item[(b)] $\mu_1^*K_{\cX_1}\leq \mu_2^* K_{\cX_2} \Leftrightarrow   
\displaystyle{ \sum_{i=1}^{n'} \frac{a_i}{r_i} \leq \frac{1}{r_{n+1}} }$.
\end{enumerate}
\end{lemma}
\begin{proof} From the above computation, 
$$
\mu_1^*K_{\cX_1}-\mu_2^* K_{\cX_2}
= \Bigl(\sum_{i=1}^n a_i'-\frac{r_{n+1}'}{r_n}\Bigr)\cD'_{n+1} 
= \Bigl(\sum_{i=1}^{n'}\frac{a_i}{r_i}  - \frac{1}{r_{n+1}}\Bigr)(r_{n+1}'\cD'_{n+1}).
$$
\end{proof}

\begin{theorem}
Let $\mu_i:\cW\to \cX_i$ be defined as above, and let 
\begin{eqnarray*}
&& F'_{12}:= \mu_{1*}\circ \mu_2^*:\Perf_T(\cX_2)\to \Perf_T(\cX_1) \\ 
&& F'_{21}:= \mu_{2*}\circ \mu_1^*:\Perf_T(\cX_1)\to \Perf_T(\cX_2)
\end{eqnarray*}
be Fourier-Mukai functors. 
\begin{enumerate}
\item[(a)] If $\mu_1^*K_{\cX_1}\geq \mu_2^* K_{\cX_2}$ then $F'_{12}$ is cohomologically full and faithful. 
\item[(b)] If $\mu_1^*K_{\cX_1}\leq \mu_2^* K_{\cX_2}$ then $F'_{21}$ is cohomologically full and faithful.
\item[(c)] If $\mu_1^*K_{\cX_1} =   \mu_2^* K_{\cX_2}$ then $F'_{12}$ and $F'_{21}$ are quasi-equivalences.
\end{enumerate}
\end{theorem}

In (c), $F_{12}$ and $F_{21}$ are not inverses of each other in general, as we will see
in the following example.
\begin{example}
$N=\bZ^2$, 
$$
\beta_1=\left[ \begin{array}{rrr}  1& -1 & 0\\ 0 & -2 & -1 \end{array} \right],\quad
\beta_2=\left[ \begin{array}{rr}  1& -1\\ 0 & -2 \end{array} \right],\quad
\beta'=\left[\begin{array}{rrr} 1& -1 & 0\\ 0 & -2 & -2 \end{array} \right].
$$
\begin{eqnarray*}
&& v_1=(1,0), \quad v_2= (-1,-2),\quad v_3= (0,-1), \\
&& r_1=r_2=r_3=1,\quad r_3'=2,\quad a_1=a_2=\frac{1}{2}, \quad a_1'= a_2'=1.
\end{eqnarray*}

$\cX_1$ is the total space of $O_{\bP^1}(-2)$, and $\cX_2=[\bC^2/\bZ_2]$.
Given $c_1, c_2, c_3\in \bZ$, we define
\begin{eqnarray*}
\cL_{1,(c_1,c_2,c_3)}&=& \cO_{\cX_1}(c_1\cD_{1,1}+c_2\cD_{1,2}+c_3\cD_{1,3})\\
\cL_{2,(c_1,c_2)}&=&\cO_{\cX_2}(c_1\cD_{2,1}+c_2\cD_{2,2}). 
\end{eqnarray*}
Then 
$$
F'_{12}(\cL_{2,(c_1,c_2)}) =
\begin{cases} \cL_{1,(c_1,c_2,\frac{c_1+c_2}{2})},  & c_1+ c_2\textup{ is even},\\
 \cL_{1,(c_1,c_2,\frac{c_1+c_2-1}{2})}, & c_1+ c_2\textup{ is odd}.
\end{cases}
$$
$$
\begin{cases}
F'_{21}(\cL_{1,(c_1,c_2,\frac{c_1+c_2}{2})})=\cL_{2,(c_1,c_2)}, & c_1+ c_2\textup{ is even},\\
F'_{21}(\cL_{1,(c_1,c_2,\frac{c_1+c_2+1}{2})}=\cL_{2,(c_1+c_2)} & c_1+ c_2\textup{ is odd}.
\end{cases}
$$
$$
\begin{cases}
F'_{12}\circ F'_{21}(\cL_{1,(c_1,c_2,\frac{c_1+c_2}{2})})=\cL_{1,(c_1,c_2,\frac{c_1+c_2}{2})}, & c_1+ c_2\textup{ is even},\\
F'_{12}\circ F'_{21}(\cL_{1,(c_1,c_2,\frac{c_1+c_2+1}{2})}=\cL_{1,(c_1, c_2, \frac{c_1+c_2-1}{2})} & c_1+ c_2\textup{ is odd}.
\end{cases}
$$

The corresponding functors for constructible sheaves are shown in Figure \ref{fig:crepant}.
\begin{figure}[h]
\begin{center}
\psfrag{+}{$\oplus$}
\psfrag{F12}{\small $F_{12}$}
\psfrag{F21}{\small $F_{21}$}
\includegraphics[scale=0.6]{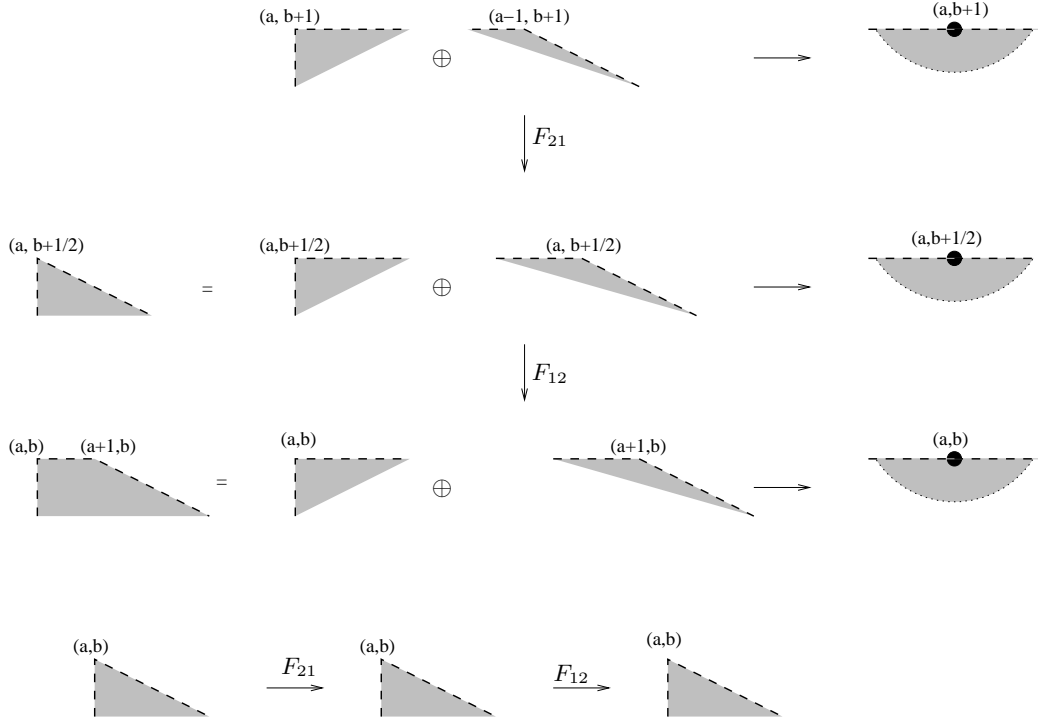}
\end{center}
\caption{$a,b$ are integers. Maps between constructible sheaves are induced by inclusion maps
of open subsets of $\bR^2$.}
\label{fig:crepant}
\end{figure}
\end{example}

\subsection{Divisorial contraction: $F'_{12}= \mu_{1*} \circ \mu_2^*$} \label{sec:caseII} 
Let $\mu_i:\cW\to \cX_i$ be defined as in Section \ref{sec:contraction}. Recall that
there is a toric morphism $f:X_1\to X_2$ which is a divisorial contraction. 
In this section, we study the Fourier-Mukai functor
$F'_{12}=\mu_{1*}\circ \mu_2^*$  and the corresponding functor
for constructible sheaves. We obtain an equivariant version of \cite[Theorem 4.2 (2)]{Ka}.

The pullback map $\mu_2^*$ has been studied in \cite{stack}. 
The identity map $id: N\to N$ induces a morphism $\bSi'\to \bSi_2$ of stacky fans, which induces 
a morphism $\mu_2 : \cW\to \cX_2$ of toric orbifolds. The following square of functors 
commutes up to natural isomorphism by \cite[Theorem 5.16]{stack}:
\begin{equation*}
\begin{CD}
\langle\Theta'_2\rangle & @>{\kappa_2}>>  & \langle \Theta_2\rangle\\
@V{\mu_2^*}VV & &  @V{id_!}VV \\
\langle\Theta'_\cW\rangle &@>{\kappa'}>> & \langle \Theta_\cW\rangle
\end{CD}
\end{equation*}
where $\kappa_2 =\kappa_{\bSi_2}$ and $\kappa'=\kappa_{\bSi'}$.
Since $id$ is the identity map and $id_!$ is cohomologically full and faithful, 
$\mu_2^*$ is also a cohomologically full and faithful functor. 

The toric orbifolds $\cW$ and $\cX_1$ have the same coarse moduli space
$X_1=X_{\bSi_1}$. The pushforward functor $\mu_{1*}$ was described
in terms of theta sheaves in Section \ref{sec:caseI}. We now describe the composition $F'_{12}=\mu_{1*}\circ \mu_2^*$ 
and the corresponding functor $F_{12}$ on constructible sheaves.

Let $\si\in \Si_2$ be a $d$-dimensional cone generated by the rays $\rho_{i_1},\dots, \rho_{i_k}$, $k=1,\dots,d$. Denote $t_k=\langle \chi, r_{i_k} v_{i_k}\rangle_{2,\si}\in\bZ$ for a given $\chi\in N_{2,\si}$.
The theta sheaf $\Theta_2(\si, \chi)\in \langle \Theta_2 \rangle$ is the costandard constructible 
sheaf supported on the submanifold given by
$$
(\sigma^\vee)^\circ_\chi=\{x\in M_\bR: \langle x, v_{i_k}\rangle > \frac{t_k}{r_{i_k}},\ k=1,\dots,d\}
$$
where $\langle\ ,\ \rangle: M_\bR\times N_\bR \to \bR$ is the natural pairing.

\begin{proposition}\label{pro:II-F}
The Fourier-Mukai functor $F_{12}'$ takes a theta sheaf in $\langle \Theta_2' \rangle$ to $\langle \Theta_1' \rangle$. Moreover, if $v_{n+1} \in  \si$, then the analogue constructible functor
$F_{12}=\kappa_1\circ F'_{12}\circ \kappa_2^{-1}: \langle \Theta_2 \rangle \to \langle \Theta_1 \rangle$ 
takes $\Theta_2(\si,\chi)$ to the costandard constructible sheaf on  
$$
F(\si_\chi^\vee)^\circ:=\{x\in M_\bR: \langle x, v_{i_k}\rangle > \frac{t_k}{r_{i_k}},\ k=1,\dots,d; 
\langle x, v_{n+1}\rangle >  \frac{t_{n+1}}{r_{n+1}} \},
$$
where 
$$
t_{n+1}= \lceil r_{n+1}\sum_{i=1}^{n'}\frac{a_k t_k}{r_k} \rceil \in \bZ.
$$
Otherwise if $v_{n+1}\notin \si$ then $F_{12}(\Theta_2(\si,\chi))=\Theta_1(\si,\chi)$.
\end{proposition}
\begin{proof}
Suppose that $\si\in \Si_2(d)$ and $v_{n+1}\in \si$. Let
$v_{i_1},\ldots, v_{i_d}$ be defined as above. Then we
may assume $i_k=k$ for $k=1,\ldots, n'$, and
$$
n'< i_{n'+1}<\cdots < i_d\leq n.
$$
We have
$$
\si=\bigcup_{k=1}^{n'} \si_k
$$
where $\si_k\in \Si_1(d)$ is the cone generated by 
$$
v_1, \ldots, v_{k-1}, v_{k+1}, \ldots,v_{n'}, v_{i_{n'+1}},\ldots, v_{i_d}, v_{n+1}.
$$
For $1\leq j_0<\ldots < j_k\leq n'$, let $\si_{j_0 \cdots j_k} =\si_{j_0}\cap \cdots \cap\si_{j_k}\in \Si_1(d-k)$,
and let $\chi_{j_0\ldots j_k}\in M'_{\si_{j_0}\ldots, \si_{j_k}}$ be the image
of $\chi\in M_{2,\si}$ under the group homomorphism $M_{2,\si}\to M'_{\si_{j_0\cdots j_k}}$.
Let $P(\chi_1,\ldots,\chi_{n'}) \in Sh_c(M_\bR;\Lambda_{\bSi'})$ be the following cochain complex:
$$
\bigoplus_{1\leq j_0\leq n'} \Theta_\cW(\si_{j_0}, \chi_{j_0}) \to 
\bigoplus_{1\leq j_0< i_1\leq n'} \Theta_\cW(\si_{j_0 j_1}, \chi_{j_0, j_1}) \to \cdots 
$$
Then $P(\chi_1,\ldots,\chi_{n'})$ is quasi-isomorphic to  
$j_{(\si_\chi^\vee)^\circ !}\bC_{(\si_\chi^\vee)^\circ}[n]$.

If $\tau,\tau'\in \Si_1$ and $\tau\subset \tau'$ then there
are surjective group homomorphisms $f_{1,\tau\tau'}^*:M_{1,\tau'}\to M_{1,\tau}$ and
$f^{\prime *}_{\tau\tau'}:M'_{\tau'}\to M'_{\tau}$.
Recall that the pushforward map $\mu_{1*,\tau}$ is the pushforward map of the characters for a single cone defined in Section \ref{sec:caseI}.
These maps are compatible with the restriction map $f^*$:
$$
\mu_{1,\tau *}\circ f_{\tau\tau'}^{\prime *} = f_{1,\tau\tau'}^*\circ \mu_{1,\tau'*}.
$$
Let  $\phi_{i_0\cdots i_k} := \mu_{1,\si_{i_0\ldots,i_k} *}(\chi_{i_0\cdots i_k})\in M_{1,\si_{i_0\cdots i_k}}$,
and let $P(\phi_1,\ldots,\phi_{n'}) \in Sh_c(M_\bR;\Lambda_{\bSi_1})$ 
be the following cochain complex:
$$
\bigoplus_{1\leq i_0\leq n'} \Theta_1(\si_{i_0}, \phi_{i_0}) \to 
\bigoplus_{1\leq i_0< i_1\leq n'} \Theta_1(\si_{i_0 i_1},\phi_{i_0, i_1}) \to \cdots 
$$
It remains to show that $P(\phi_1,\ldots,\phi_{n'})$
is quasi-isomorphic to $j_{F(\si_\chi^\vee)^\circ !}\bC_{F(\si_\chi^\vee)^\circ}[n]$.
It suffices to prove the following two statements:
\begin{enumerate}
\item[(i)] The piecewise linear function $\psi:\si\to \bR$
defined by $\phi_i\in M_{1,\si_i}$ for $i=1,\dots,{n'}$ is  convex:
$$
\psi(v_{n+1})\geq \sum_{i=1}^{n'} a_i \psi(v_i).
$$
\item[(ii)] $F(\si_\chi^\vee)^\circ  = \{ x\in M_\bR\mid \langle x, v\rangle > \psi(v) \text{ for any $v\in\si \subset N_\bR$\}}$.
\end{enumerate}
(i) and (ii) follow from:  
\begin{eqnarray*}
&& \psi(v_k) = \frac{t_k}{r_k},\quad k=1,\ldots, n', i_1,\ldots, i_{d-n'}, n+1\\
&& \frac{t_{n+1}}{r_{n+1}} =
\frac{1}{r_{n+1}} \lceil  r_{n+1} \sum_{i=1}^{n'}\frac{a_i t_i}{r_i}\rceil
\geq \sum_{i=1}^{n'}a_i \frac{t_i}{r_i}.
\end{eqnarray*}
\end{proof}

\begin{proposition}\label{pro:II-hom}
Suppose that $\displaystyle{\sum_{i=1}^n \frac{a_i}{r_i}\geq \frac{1}{r_{n+1}} }$.
We have the following statements involving the map $F$:
\begin{enumerate}
\item $(\si_\chi^\vee)^\circ \subset (\si_{\chi'}^{\prime \vee})^\circ \Rightarrow 
F(\si_\chi^\vee)^\circ\subset F(\si_{\chi'}^{\prime \vee})^\circ$.
\item $(\si_\chi^\vee)^\circ \not \subset (\si_{\chi'}^{\prime \vee})^\circ
\Rightarrow F(\si_\chi^\vee)^\circ\not \subset F(\si_{\chi'}^{\prime \vee})^\circ$ and
$F(\si_{\chi}^{\vee})^\circ - F(\si_{\chi'}^{\prime \vee})^\circ$ is contractible.
\end{enumerate}
\end{proposition}
\begin{proof}
(1) We only need to show the case $\si$ and $\si'$ both contain $\rho_1,\dots, \rho_{n'}$. It is obvious that $\si\supset \si'$. Let $v_1,\dots, v_{i_d}$ be the generators of $\si$, and $v_1,\dots, v_{i_{d'}}$ generate $\si'$ where $1\le d'\le d$. Similarly to the definition of $t_k$, set $t_k'=\langle \chi',  r_{i_k} v_{i_k}\rangle_{2,\si}$, and 
$$
t'_{n+1}=\lceil r_{n+1}\sum_{i=1}^{n'}\frac{a_k t'_k}{r_k} \rceil.
$$ 
The inclusion $(\si_\chi^\vee)^\circ \subset (\si_{\chi'}^{\prime \vee})^\circ$ gives $t_k\ge t_k'$, and a straightforward calculation shows $t_{n+1}\ge t_{n+1}'$. It follows that $F(\si_\chi^\vee)^\circ\subset F(\si_{\chi'}^{\prime\vee} )^\circ$ by definition.

(2) If $\si\not\supset \si'$ the statement is trivial. In case that $\si\supset \si'$, we must have some $k_0$ such that $t_{i_{k_0}}<t_{i_{k_0}}'$, which followed by $F(\si_\chi^\vee)^\circ\not \subset F(\si_{\chi'}^{\prime \vee})^\circ$. The only situation that $F(\si_\chi^{\vee})^\circ-F(\si_{\chi'}^{\prime\vee})^\circ$ is not contractible is that $t_k'>t_k$ while $t'_{n+1}=t_{n+1}$, but this is impossible since
\begin{eqnarray*}
t_{n+1}'-t_{n+1}&=& \lceil r_{n+1}\sum_{i=1}^{n'}\frac{a_k t'_k}{r_k} \rceil-
\lceil r_{n+1}\sum_{i=1}^{n'}\frac{a_k t_k}{r_k} \rceil\\
 &\geq & \lceil r_{n+1}\sum_{i=1}^{n'}\frac{a_k (t'_k-t_k)}{r_k}\rceil
\geq  \lceil r_{n+1}\sum_{i=1}^{n'}\frac{a_k}{r_k}\rceil
\geq  \lceil 1 \rceil > 0. 
\end{eqnarray*}
\end{proof}

Proposition \ref{pro:II-F} and Proposition \ref{pro:II-hom} give the following theorem:
\begin{theorem}
If $\mu_1^*K_{\cX_1}\geq \mu_2^*K_{\cX_2}$, or equivalently,
$$
\sum_{i=1}^n \frac{a_i}{r_i}\geq \frac{1}{r_{n+1}},
$$ 
then the Fourier-Mukai functor  
$F'_{12}=\mu_{1*}\circ\mu_2^*: \langle \Theta'_2 \rangle \to \langle \Theta'_1 \rangle$ 
is a quasi-embedding. If restricted on the full dg subcategory $\Coh_T (\cX_2)$, $F_{12}'$ is a quasi-embedding of $\Coh_T(\cX_2)$ into $\Coh_T(\cX_1)$.
\end{theorem}
\begin{proof}
Passing to constructible sheaves via CCC, it suffices to work on the constructible theta sheaves since they are generators. The theorem follows from the simple facts
$$
\Ext^*(\Theta_2(\si,\chi),\Theta_2(\si',\chi')) =
\begin{cases}
\bC[0] &\text{if $(\si_\chi^\vee)^\circ \subset (\si_{\chi'}'^\vee)^\circ$},\\
0 &\text{if $(\si_\chi^\vee)^\circ \not\subset (\si_{\chi'}'^\vee)^\circ$},
\end{cases}
$$
and
\begin{eqnarray*}
&&
\Ext^*(i_{F(\si_\chi^\vee)^\circ!}(\si_\chi^\vee)^\circ,i_{F(\si_{\chi'}^{\prime \vee})^\circ!}(\si_{\chi'}^{\prime \vee})^\circ)\\
&=&\begin{cases}
\bC[0] & \text{if $F(\si_\chi^\vee)^\circ \subset F(\si_{\chi'}'^\vee)^\circ$,}\\
0 &\text{\parbox[t]{8cm}{$F(\si_\chi^\vee)^\circ\not \subset F(\si_{\chi'}^{\prime\vee})^\circ$ and $F(\si^\vee_\chi)^\circ-F(\si^{\prime \vee}_{\chi'})^\circ$ is contractible.}}
\end{cases}
\end{eqnarray*}

\end{proof}


\begin{figure}[h]
\psfrag{F12}{$F_{12}$}
\psfrag{(a,b)}{\footnotesize $(a,b)$}
\psfrag{(a-1/2,b)}{\footnotesize$(a-\frac{1}{2},b)$}
\psfrag{(a-1/2,b+1/2)}{\footnotesize$(a-\frac{1}{2},b+\frac{1}{2})$}
\includegraphics[scale=0.35]{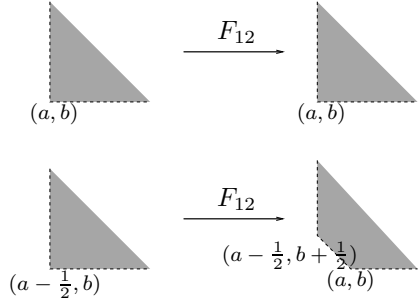}
\caption{$a,b$ are integers. Constructible sheaves are costandard constructible sheaves over shaded regions. Maps between constructible sheaves are induced by inclusion maps
of open subsets of $\bR^2$.}
\label{fig:caseII}
\end{figure}
\begin{example}\label{ex:discrepancy}
$N=\bZ^2$, 
$$
\beta_1=\left[ \begin{array}{rrr} 2 & 0 & 1\\ 0 & 1 & 1 \end{array}\right],\quad
\beta_2 =\left[\begin{array}{rr} 2 & 0 \\0& 1 \end{array}\right],\quad
\beta'=\left[\begin{array}{rrr} 2 & 0 & 2\\0 & 1 & 2 \end{array}\right].
$$
$\cX_2= [\bC/\bZ_2]\times \bC$, $X_2=\bC^2$.
\begin{eqnarray*}
&& v_1=(1,0),\quad v_2=(0,1), \quad v_3 =(1,1), \\
&& r_1=2, \quad r_2=r_3=1, \quad r_3'=2,\quad a_1=a_1'=1 ,\quad a_2=1,\quad a_2'=2.
\end{eqnarray*}
For $c_1, c_2,c_3\in \bZ$, define
\begin{eqnarray*}
\cL_{1,(c_1,c_2,c_3)}&=& \cO_{\cX_1}(c_1\cD_{1,1}+c_2\cD_{1,2}+ c_3\cD_{1,3})\\
\cL_{2,(c_1,c_2)}&=& \cO_{\cX_2}(c_1\cD_{2,1}+c_2\cD_{2,2}) =p_1^* \cO_{\bC^2}(\frac{c_1}{2}D_{2,1}+ D_{2,2})
\end{eqnarray*}
Then 
$$
F'_{12}(\cL_{2,(c_1,c_2)})=\begin{cases}
\cL_{1,(c_1, c_2, \frac{c_1}{2} + c_2)}, & c_1\textup{ is even},\\
\cL_{1, (c_1, c_2,\frac{c_1-1}{2} + c_2)}, & c_1 \textup{ is odd}.
\end{cases}
$$
The corresponding functors for constructible sheaves are shown in Figure \ref{fig:caseII}.
\end{example}

\subsection{Divisorial contraction: $F'_{21}=\mu_{2*}\circ \mu_1^*$} \label{sec:caseIV} 
Let $\mu_i:\cW\to \cX_i$ be defined as in Section \ref{sec:contraction}
and Section \ref{sec:caseII}. In this section, we study the Fourier-Mukai functor
$F'_{21}=\mu_{2*}\circ \mu_1^*$  and the corresponding functor
for constructible sheaves. We obtain an equivariant version of \cite[Theorem 4.2 (4)]{Ka}.

The pullback map $\mu_1^*$ has been studied in \cite{stack}. 
The identity map $id: N\to N$ induces a morphism $\bSi'=(N,\Si_1,\beta')\to \bSi_1=(N,\Si_1,\beta_1)$ 
of stacky fans, which induces  a morphism $\mu_1 : \cW\to \cX_1$ of toric orbifolds. 
The following square of functors  commutes up to natural isomorphism by \cite[Theorem 5.16]{stack}:
$$
\begin{CD}
\langle\Theta'_1\rangle & @>{\kappa_1}>>  & \langle \Theta_1\rangle\\
@V{\mu_1^*}VV & &  @V{{id}_!}VV \\
\langle\Theta'_\cW\rangle &@>{\kappa'}>> & \langle \Theta_\cW\rangle
\end{CD}
$$
where $\kappa_1 =\kappa_{\bSi_1}$ and $\kappa'=\kappa_{\bSi'}$.
Since $id$ is the identity map and $id_!$ is cohomologically full and faithful, 
$\mu_1^*$ is also a cohomologically full and faithful functor. 

It remains to study the pushforward map $\mu_{2*}$. Generally speaking, the
image of $\cQ^\fin_T(\cW)= \langle \Theta'_\cW\rangle$
under the pushforward map $\mu_{2*}:\cQ_T(\cW)\to \cQ_T(\cX_2)$ is
{\em not} contained in $\cQ^\fin_T(\cX_2)= \langle \Theta'_2\rangle$. 

\subsubsection{Notation}
By definition, 
$$
N_{2,\si_{X_2}}=\bigoplus_{i=1}^n \bZ b_i \subset N_\bR.
$$ 
Define
$b^*_1,\ldots, b^*_n\in M_\bR$ by $\langle b_i^*, b_j\rangle =\delta_{ij}$. Then
the dual lattice of $N_{2,\si_{X_2}}$ is  given by  
$$
M_{2,\si_{X_2}}=\bigoplus_{i=1}^n \bZ b_i^* \subset M_\bR.
$$ 

Recall from Section \ref{sec:contraction} that
$$
v_{n+1}=\sum_{i=1}^{n'} a_i v_i,\quad
b_{n+1}= \sum_{i=1}^{n'} \alpha_i b_i,\quad
b_{n+1}' = m b_{n+1}=\sum_{i=1}^{n'}\beta_i b_i,
$$
where 
$$
a_i\in \bQ_{>0}, \quad \alpha_i =\frac{r_{n+1}}{r_i}a_i\in \bQ_{>0}, \quad
m\in \bZ_{>0},\quad
\beta_i = m \alpha_i \in \bZ_{>0}
$$
for $i=1,\ldots,n'$.  We fix the following notation:
\begin{itemize}
\item Let $\bar I=\{1,\dots, n+1\},\ I=\{1,\dots, n\},\ I'=\{1,\dots, n'\}$.
\item Given a proper subset $J$ of $\bar I$, let
$\si_J$ denote the cone in $N_\bR$ generated
by 
$$
\{ v_j\mid j\in J\}.
$$ 
In particular, $\si_\emptyset =\{0\}$,  where $\emptyset$ is the empty set. 
\end{itemize}
With the above notation, we have
$$
\si_{X_2}=\si_I,\quad \si_{i_0}=\si_{\bar I-\{ i_0\}}\textup{ if }i_0\in I',
$$
$$
\Si_1 = \{ \si_J\mid I'\not \subset J\subset \bar I\},\quad
\Si_2=\{\si_J\mid J\subset I\}.
$$
We define a map $\Lambda:=\{ J\subset \bar I\mid I'\not \subset
J\}\to 2^I=\{ J'\subset I\}$, $J\mapsto J'$, such that
$\si_{J'}\in \Si_2$ is the intersection of all 
cones in $\Si_2$ which contains $\si_J\in \Si_1$. More explicitly,
$$
J'=\begin{cases}
J &\textup{if } n+1\notin J,\\
(J-\{n+1\})\cup I' & \textup{if } n+1\in J.
\end{cases}
$$
For $J\in \Lambda$, define
$\cX_{1,J}=\cX_{1,\si_J}$ and $\cW_J=\cW_{\si_J}$;
for $J\in 2^I$, define $\cX_{2,J}=\cX_{2,\si_J}$.

\begin{lemma} Suppose that $J\in \Lambda$, so that $\si_J\in \Si_1$.
If $n+1\notin J$ then
$$
F'_{21}\Theta'_1(\si_J,\phi)=\Theta'_2(\si_J,\phi)
$$
for any $\phi\in M_{1,\si_J}=M_{2,\si_J}$.
\end{lemma}
\begin{proof} If $n+1\notin J$ then
$\mu_i:\cW\to \cX_i$ restricts to the identity map
$\cW_J\to \cX_{i,J}$, $i=1,2$.
\end{proof}
We will consider the case $n+1\in J$ later.

By definition, $\displaystyle{N_{1,\si_{i_0}}=\bigoplus_{i\in \bar I-\{i_0\}}\bZ b_i}$.
Since the cones $\si_{X_2}$ and $\si_{i_0}$ are $n$-dimensional, one may regard $M_{2,\si_{X_2}}$ and $M_{1,\si_{i_0}}$ as subsets in $M_\bR$, embedded in a canonical way. Straightforward calculations show that
\begin{lemma}
\begin{eqnarray*}
M_{1,\si_{i_0}} &=& \bigoplus_{i\in I'-\{i_0\}}\bZ(b_i^*-\frac{\alpha_i}{\alpha_{i_0}}b^*_{i_0})\oplus \bigoplus \bZ \frac{b_{i_0}^*}{\alpha_{i_0}}
\oplus \bigoplus_{i\in I-I'} \bZ b_i^*.\\
M'_{\si_{i_0}}&=&  \bigoplus_{i\in I'-\{i_0\}}\bZ(b_i^*-\frac{\alpha_i}{\alpha_{i_0}}b^*_{i_0})\oplus \bigoplus \bZ \frac{b_{i_0}^*}{m\alpha_{i_0}}
\oplus \bigoplus_{i\in I-I'} \bZ b_i^*.
\end{eqnarray*}
\end{lemma}

\subsubsection{Reduction} We fix $i_0\in I'$.
Define $\mu_{1,i_0}:=\mu_1|_{\cW_{\si_{i_0}}}:\cW_{\si_{i_0}}\to \cX_{1,\si_{i_0}}$ and 
$\mu_{2,i_0}:=\mu_2|_{\cW_{\si_{i_0}}}:\cW_{\si_{i_0}}\to \cX_{2,\si_{X_2}}=\cX_2$.
Define  $F'_{21,i_0}=\mu_{2,i_0 *}\circ \mu_{1,i_0}^*$. Suppose that
$J\in \Lambda$ and $\si_J\subset \si_{i_0}$. Let $j:\cX_{1,J}\to \cX_{1,\si_{i_0}}$
be the open embedding. For every $\phi\in M_{1,\si_J}$, define
$\Theta_{1,i_0}(\si_J,\phi):= j_*\cO_{\cX_{1,J}}(\phi) \in \cQ_T^\fin(\cX_{1,\si_{i_0}})$,
Then
$$
F_{21}\Theta'_1(\si_J,\phi) =j_{\si_{i_0*}}F_{21,i_0}\Theta'_{1,i_0}(\si_J,\phi),
$$
where $j_{\si_{i_0}}: \cX_{1,\si_{i_0}}\hookrightarrow \cX_1$ is the embedding of $\cX_{1,\si_{i_0}}$. 
Every $\si_J\in \Si_1$ is contained in $\si_{i_0}$ for some $i_0\in I'$, so it suffices
to describe $F_{12,i_0}$ for any $i_0\in I'$. 

\subsubsection{Coordinate rings}
For some $i_0\in I'$, we define the following notations.
\begin{itemize}
\item For $i\in I'-\{i_0\}$, define 
$$
x_i=\chi^{b_i^*}\in \bC[M_{2,\si_{X_2}}],\quad 
y_i=\chi^{b_i^*-\frac{\alpha_i}{\alpha_{i_0}} b_{i_0}^*}\in \bC[M_{1,\si_{i_0}}],
$$ 
where $\chi^{b_i^*}$ is defined as in \cite[Section 1.3]{Fu}. 
\item Define
$$
y_{i_0}=\chi^{ \frac{b^*_{i_0}}{\alpha_{i_0} } } \in \bC[M_{1,\si_{i_0}}],\quad
z=\chi^{\frac{b^*_{i_0}}{m\alpha_{i_0}} } \in \bC[M'_{\si_{i_0}}].
$$
In particular, $y_{i_0}=z^m$.
\item For $i\in I-I'$, i.e. $n'+1\leq i\leq n$,  define
$y_i = \chi^{b_i^*}$.
\item Define rings
\begin{eqnarray*}
A_1&:=& \bC[ \si_{i_0}^\vee \cap M_{1,\si_{i_0}}]=\bC[y_1,\ldots,y_n],\\
A'&:=& \bC[\si_{i_0}^\vee \cap M'_{\si_{i_0}}]=\bC[y_1,\ldots, y_{i_0-1},z, y_{i_0+1},y_n],\\
A_2&: =& \bC[\si_{X_2}^\vee\cap M_{2,\si_{X_2}}]=\bC[x_1,\ldots,x_{n'}, y_{n'+1},\ldots, y_n ],
\end{eqnarray*}
\end{itemize}

We define
$$
U_1=\Spec A_1, \quad U'=\Spec A', \quad U_2=\Spec A_2.
$$
Then $U_1$, $U'$, and $U_2$ are isomorphic to $\bC^n$. 
Define
$$
\tT_1=\Spec\bC[M_{1,\si_{i_0}}],\quad
\tT'=\Spec\bC[M'_{\si_{i_0}}],\quad
\tT_2=\Spec\bC[M_{2,\si_{X_2}}],\quad
T=\Spec\bC[M].
$$
Then $\tT_1$, $\tT'$, $\tT_2$, and $T$ are isomorphic to $(\bC^*)^n$.
$\tT_1$, $\tT'$, and $\tT_2$ act on $U_1$, $U'$, and $U_2$,
respectively.

There are short exact sequence of abelian groups
$$
1\to G_1\to \tT_1 \to T\to 1,\quad
1\to G'\to \tT'\to T\to 1,\quad
1\to G_2\to \tT_2\to T\to 1.
$$
where $G_1$, $G'$, and $G_2$ are finite groups.
We have
$$
\cX_{1,\si_{i_0}}=[U_1/G_2],\quad \cW_{\si_0}=[U'/G'], \quad\cX_2= [U_2/G_2].
$$
The morphism $\mu_{1,i_0}:\cW_{\si_{i_0}}\to  \cX_{1,\si_{i_0}}$ lifts to
$g_1: U'\to U_1$, where
$$
g_1(y_1,\ldots, y_{i_0-1}, z, y_{i_0+1}, \ldots, y_n)=
(y_1,\ldots, y_{i_0-1}, z^m, y_{i_0+1}, \ldots, y_n).
$$
The morphism $\mu_{2,i_0}:\cW_{\si_{i_0}}\to \cX_2$ lifts to
$g_2:U'\to U_2$, where
\begin{eqnarray*}
&& g_2(y_1,\ldots, y_{i_0-1}, z, y_{i_0+1}, \ldots, y_n)\\
&=&(y_1 z^{\beta_1},\ldots, y_{i_0-1}z^{\beta_{i_0-1}}, z^{\beta_{i_0}}, 
y_{i_0+1} z^{\beta_{i_0+1}}, \ldots, y_{n'} z^{\beta_{n'}},  y_{n'+1},\ldots, y_n)
\end{eqnarray*}

Suppose that $J\in \Lambda$. Then $J\subset \si_{i_0}$ for 
some $i_0\in I'$. We fix $i_0$ and $J$, and assume that $n+1\in J$.
Define
$$
K_1= \bar I-J-\{i_0\},\quad
K_2 = I- J'.
$$
Define rings
$$
B_1= A_1[y_i^{-1}]_{ i \in K_1},\quad
B' = A'[y_i^{-1}]_{i \in K_1},\quad
B_2=A_2[y_i^{-1}]_{i \in K_2}.
$$
where $A_1[y_i^{-1}]_{i\in K_1}$ is the ring $A_1$ adjoint with $y_i^{-1}$ for all $i\in K_1$,  etc.

We define
$$
V_1=\Spec B_1,\quad V'=\Spec B',\quad V_2=\Spec B_2.
$$
The inclusion $A_1\subset B_1$, $A'\subset B'$, and $A_2\subset B_2$ induce
open embeddings
$$
V_1\subset U_1,\quad V'\subset U',\quad V_2\subset U_2.
$$
We have
$$
\cX_{1,J}=[V_1/G_1],\quad \cW_J=[V'/G'],\quad \cX_{2,J'}=[V_2/G_2].
$$

\subsubsection{Sheaves and modules}
$\Theta'_{1,i_0}(\si_J,\phi)$ corresponds to
a $\tT_1$-equivariant quasicoherent sheaf $\widetilde{\Theta}'_{1,i_0}(\si_J,\phi)$
on $U_1$, and $F_{21,i_0}\Theta'_{1,i_0}(\si_J,\phi)$
corresponds to the $\tT_2$-equivariant quasicoherent sheaf
$g_{2*} g_1^* \widetilde{\Theta}'_{1,i_0}(\si_J,\phi)$ on $U_2$.
Let $H$ be the kernel of $\tT'\to \tT_2$. Define
\begin{eqnarray*}
&& Q_1 = \Gamma(U_1, \widetilde{\Theta}'_{1,i_0}(\si_J,\phi)),\quad
Q'= \Gamma(U',g_1^*\widetilde{\Theta}'_{1,i_0}(\si_J,\phi)),  \\
&& Q_2 = \Gamma(U_2, g_{2*} g_1^* \widetilde{\Theta}'_{1,i_0}(\si_J,\phi) )=\Gamma(U',g_1^*\widetilde{\Theta}'_{1,i_0}(\si_J,\phi) )^H,\\
\end{eqnarray*}
Then
\begin{enumerate}
\item $\widetilde{\Theta}'_{1,i_0}(\si_J,\phi)$ is the $\tT_1$-equivariant quasicoherent sheaf 
on $U_1$ defined by the $\tT_1$-equivariant $A_1$-module $Q_1$.
\item $g_1^* \widetilde{\Theta}'_{1,i_0}(\si_J,\phi)$ is the $\tT'$-equivariant quasicoherent sheaf 
on $U'$ defined by the $\tT'$-equivariant $A'$-module $Q'$.
\item $g_{2*} g_1^*\widetilde{\Theta}'_{1,i_0}(\si_J,\phi)$ is the $\tT_2$-equivariant quasicoherent sheaf 
on $U_2$ defined by the $\tT_2$-equivariant $A_2$-module $Q_2$.
\end{enumerate}

More explicitly, $\phi\in M_{1,\si_J}$ is determined by  $c_i=\langle \phi,b_i\rangle \in \bZ$,  $i\in J$. We have
\begin{eqnarray*}
Q_1&=& \bC[ (\si_J)_\phi\cap M_{1,\si_{i_0}}]= y_{i_0}^{c_{n+1}}\cdot (\prod_{j\in J-\{n+1\}}y_j^{c_j})\cdot B_1\\
Q'&=& \bC[ (\si_J)_\phi\cap M'_{\si_{i_0}}]=z^{mc_{n+1}}\cdot(\prod_{j\in J-\{n+1\}}y_j^{c_j})\cdot B' \\
Q_2&=& \bC[ (\si_J)_\phi\cap M'_{\si_{i_0}}] \cap \bC[M_{2,\si_{X_2}}]\\
&=& x_{i_0}^{\frac{c_{n+1}}{\alpha_{i_0}}}\cdot \prod_{j\in J\cap I'} (x_j x_{i_0}^{-\alpha_j/\alpha_{i_0} })^{c_j}\cdot \prod_{j\in  J'-I'} y_j^{c_j} \cdot 
g(B_1) \cap \bC[M_{2,\si_{X_2}}]
\end{eqnarray*}
where 
\begin{eqnarray*}
g(B_1)&=&\bC[x_j x_{i_0}^{-\alpha_j/\alpha_{i_0}} ]_{j\in J\cap I'}\otimes_{\bC} \bC[x_{i_0}^{\frac{1}{m\alpha_{i_0}}   }]
\otimes_{\bC} \bC[x_j x_{i_0}^{-\alpha_j/\alpha_{i_0}} , x_j^{-1} x_{i_0}^{\alpha_j/\alpha_{i_0}} ]_{j\in K_1\cap I' }\\
&& \otimes_{\bC}\bC[y_j]_{j\in J'-I'} \otimes_{\bC} \bC[y_j, y_j^{-1}]_{j\in K_2}\\
\bC[M_{2,\si_{X_2}}] &=& \bC[x_1, x_1^{-1},\ldots, x_{n'}, x_{n'}^{-1}, y_{n'+1}, y_{n'+1}^{-1},
\ldots, y_n, y_n^{-1}].
\end{eqnarray*}
Here we use $z=x_{i_0}^{\frac{1}{m\alpha_{i_0}}}$ and $y_i=x_i x_{i_0}^{-\frac{\alpha_i}{\alpha_{i_0}}}$ for $i\in I'-\{i_0\}$.

Finally, we remark that
\begin{enumerate}
\item $Q_1$ is a free $B_1$-module of rank 1, and
defines a line bundle $\cO_{V_1}(\phi)$ on $V_1=\Spec B_1$.
\item $Q'$ is a free $B'$-module of rank 1, and
defines a line bundle $\cO_{V'}(\phi)$ on $V'=\Spec B'$.
\item $Q_2$ is a $B_2$-module, and defines a quasicoherent sheaf on $V_2=\Spec  B_2$.
\end{enumerate}
\subsubsection{Koszul resolution}
$Q_2$ is not finitely generated as a $B_2$-module.
The goal of this section is to find a resolution of $Q_2$ by free
$B_2$-modules. The following observations are useful:
\begin{enumerate}
\item[(i)] Let $B=\bC[x_{i_0}, y_{n'+1},\ldots, y_n]\otimes_{\bC}\bC[y_j^{-1}]_{j\in K_2}$. Then
$B$ is a subring of $B_1$, so $Q_2$ can be viewed as a $B$-module.
 We observe that $Q_2$ is a {\em free} $B$-module.
\item[(ii)] $B_2= B[x_j]_{j\in I'-\{i_0\}}$ can be viewed as a $B$-module. 
We have the following exact sequence of $B$-modules (the Koszul complex):
\begin{eqnarray*}
0&\to&( \prod_{j\in I'-\{i_0\}} x_j) B_1
\longrightarrow \cdots \longrightarrow \bigoplus_{i,j\in I'-\{ i_0\}, i<j} x_i x_j B_1\\
&&\longrightarrow \bigoplus_{j\in I'-\{i_0\} } x_j B_1
\longrightarrow B_1\longrightarrow B\to 0.
\end{eqnarray*}
\end{enumerate}

Note that
\begin{itemize}
\item The set  $I'-\{i_0\}$ is the disjoint union of $I'\cap J$ and $I'\cap K_1$. 
\item The set $J'$ is the disjoint union of $I'$ and $J'-I'$.
\end{itemize}

For any $\mathfrak{m}=(m_i)_{i\in I'-\{i_0\}}$, where
$m_i\in  \bZ_{\geq 0}$ if $i\in  I'\cap J$ and $m_i\in \bZ$
if $I'\cap K_1$,
We define $\gamma(\mathfrak m)\in M_{2,\si_{J'}}$ as follows.
\begin{eqnarray*}
\gamma(\mathfrak m)& :=& \lceil \frac{c_{n+1}}{\alpha_{i_0}}-\frac{1}{\alpha_{i_0}}(\sum_{i\in I'\cap J}\alpha_i (c_i+m_i)+
\sum_{i\in I'\cap K_1} \alpha_i m_i)\rceil b^*_{i_0} \\
&& + \sum_{i\in I'\cap J}(c_i+m_i)b^*_i+\sum_{i\in I'\cap K_1} m_i b^*_i+\sum_{i\in J'-I'} c_i b^*_i.
\end{eqnarray*}

Define
$$
\Gamma:=\{\gamma(\mathfrak m)\mid m_i\in \bZ_{\geq 0}\textup{ if   }i\in I'\cap J;\ m_i \in \bZ \textup{ if }i\in I'\cap K_1 \} \subset
M_{2,\si_{J'}}.
$$
For any $\chi=\sum_{j \in J'} k_j b_j^*\in M_{2,\si_{J'}}$, denote the monomial 
$$
f_\chi=\prod_{j\in I'} x_j^{k_j}\cdot \prod_{j\in J'-I'} y_j^{k_j}.
$$
Then $Q_2$ is a free $B$-module generated by 
$\{ f_{\chi}\mid \chi\in \Gamma\}$:
$$
Q_2=\bigoplus_{\chi\in \Gamma} f_{\chi}B.
$$

Multiplying the exact sequence in (ii) by $f_\chi$, and taking the direct sum over all $f_\chi$ for $\chi \in \Gamma$, one arrives at the following resolution of $Q_2$ by free $B_2$-modules:
\begin{eqnarray*}
0\longrightarrow \bigoplus_{\chi\in \Gamma} (\prod_{i\in I'-\{i_0\}}x_i) f_\chi B_2 \longrightarrow \dots \longrightarrow \bigoplus_{\chi\in \Gamma} \bigoplus_{i,j\in I'-\{i_0\},\ i < j} x_i x_j f_\chi B_2 \\
\longrightarrow \bigoplus_{\chi\in \Gamma} \bigoplus_{i\in I'-\{i_0\}} x_i f_\chi B_2
\longrightarrow \bigoplus_{\chi\in \Gamma} f_\chi B_2\longrightarrow Q_2 \longrightarrow 0.
\end{eqnarray*}

\subsubsection{Resolution by theta sheaves}
\begin{lemma}
The Fourier-Mukai transformed sheaf $F_{21}\Theta'(\si_J,\phi)$ admits the following resolution
\begin{eqnarray*}
0&\to& \bigoplus_{\chi\in \Gamma} \Theta'_2(\si_{J'},\chi+\sum_{i\in I'-\{i_0\}}b_i^*) \longrightarrow \dots \longrightarrow \bigoplus_{\chi\in \Gamma} \bigoplus_{i,j\in I'-\{i_0\},\ i < j} \Theta'_2(\si_{J'},\chi+b_i^*+b_j^*) \\
&& \longrightarrow \bigoplus_{\chi\in \Gamma} \bigoplus_{i\in I'-\{i_0\}} \Theta'_2(\si_{J'},\chi+b_i^*)\longrightarrow \bigoplus_{\chi\in \Gamma} \Theta'_2(\si_{J'},\chi)\longrightarrow F'_{21} \Theta'_1(\si_J,\phi) \to 0.
\end{eqnarray*}
\end{lemma}

Taking the coherent-constructible correspondence functor $\kappa$ to the resolution, we obtain a chain complex of constructible theta sheaves on $M_\bR$
\begin{eqnarray*}
0 &\to & \bigoplus_{\chi\in \Gamma} \Theta_2(\si_{J'},\chi+\sum_{i\in I'-\{i_0\}}b_i^*) \longrightarrow \dots \longrightarrow \bigoplus_{\chi\in \Gamma} \bigoplus_{i,j\in I'-\{i_0\},\ i <  j} \Theta_2(\si_{J'},\chi+b_i^*+b_j^*) \\
&& \longrightarrow \bigoplus_{\chi\in \Gamma} \bigoplus_{i\in I'-\{i_0\}} \Theta_2(\si_{J'},\chi+b_i^*)\longrightarrow 
\bigoplus_{\chi\in \Gamma} \Theta_2(\si_{J'},\chi)\longrightarrow.
\end{eqnarray*}
Although this complex is not finitely-generated by $\Theta_2$-sheaves on $M_\bR$ since it involves countably-many direct sums, it is a constructible sheaf on $M_\bR$. Thus we have obtained a functor denoted by $F_{21}:\langle\Theta_1\rangle \to Sh_{c}(M_\bR;\Lambda_{\bSi_2})$.

For the given $\si_J$ and $\phi\in M_{1,\si_J}$, define the ``Fourier-Mukai transformed set'' 
$$
F(\si_J)^\vee_\phi=\bigcup_{\chi\in \Gamma}(\si_{J'})^\vee_\chi,
$$
while similarly we denote $F((\si_J)^\vee_\chi)^\circ$ to be the interior of the above set. (In case that $n+1\not\in J$, we simply set $F(\si_J)^\vee_\phi=(\si_J)^\vee_\phi$). We have the following proposition characterizing $F_{21}(\Theta_1(\si_J,\phi))$.
\begin{proposition}

$$
F_{21}( \Theta_1(\si_J,\phi)) \cong i_!\omega_{F((\si_J)^\vee_\chi)^\circ},
$$
where $i:((\si_J)^\vee_\chi)^\circ\hookrightarrow M_\bR$ is the embedding of the open subset, and 
$\omega_{F((\si_J)^\vee_\chi)^\circ}$ is the costandard constructible sheaf on this set.
\end{proposition}
\label{prop:III-F}
\begin{proof}
In order to prove the resolution of $F_{21}(\Theta_1(\si_J,\phi))$ is quasi-isomorphic to $\omega_{F((\si_J)^\vee_\chi)^\circ}$, we only need to show they are quasi-isomorphic at every stalk $p\in M_\bR$. If $p\not\in F((\si_J)^\vee_\chi)^\circ$ the stalk of the costandard sheaf $(i_!\omega_{F((\si_J)^\vee_\chi)^\circ})_p=0$, while the stalk $(F_{21}(\Theta_1(\si_J,\phi)))_p$ is also a zero complex. It remains to show that when $p\in F((\si_J)^\vee_\chi)^\circ$, the stalk $(F_{21}(\Theta_1(\si_J,\phi)))_p$ is quasi-isomorphic to $(i_!\omega_{F((\si_J)^\vee_\chi)^\circ})_p\cong \bC[0]$.

Let $p=\sum_{i=1}^n p_i b_i^*\in M_\bR$, and $\Gamma(p)=\{\chi\in\Gamma\mid p\in((\si_{J'})^\vee_\chi)^\circ\}$. 
Set $m^0_i=\lceil p_i \rceil -1 - c_i$ for $i\in I'\cap J$, and $m^0_i=\lceil p_i\rceil -1$ if $i\in I'\cap K_1$. The character $\gamma^0:=\gamma(m^0_1,\dots,m^0_{i_0-1},m^0_{i_0+1},\dots, m^0_{n'})$ is the unique element in $\Gamma(p)$ such that 
$\gamma^0-b_i^*\not\in \Gamma(p),\ \text{for any $i\in I'-\{i_0\}$}.$ 
For any $\chi\in \Gamma(p)-\{\gamma^0\}$, denote $$I'_\chi=\{i\in I'-\{i_0\}\mid  \chi-b_i^*\in \Gamma(p)).$$ With these notations, the last two terms of the stalk $(F_{21}(\Theta_1(\si_J,\phi)))_p$ is
$$
\longrightarrow \bigoplus_{\chi\in\Gamma(p)-\{\gamma^0\}} \bigoplus_{i\in I'_\chi} \bC_{\chi} \longrightarrow \bigoplus_{\chi\in\Gamma(p)} \bC_\chi \longrightarrow,
$$
where $\bC_\chi\cong\bC$ is indexed by the character $\chi$. The image of the middle arrow is $\bigoplus_{\chi\in \Gamma(p)-\{\gamma^0\}}\bC_\chi$. Thus one defines a chain map $q:(F_{21}(\Theta_1(\si_J,\phi)))_p\to \bC[0]$ where
\begin{eqnarray*}
q_0: \bigoplus_{\chi\in\Gamma(p)} \bC_\chi &\to& \bC\\
\    (k_\chi)_{\chi\in \Gamma(p)}         &\mapsto& \sum_{\chi\in \Gamma(p)} k_\chi,
\end{eqnarray*}
and $q_j=0$ for $j\ne 0$. This map induces the cohomology map $H^*(q)=id:\bC[0]\to\bC[0]$, and it is a quasi-isomorphism.
\end{proof}

\begin{remark}
Although the definition of $F(\si_J)^\vee_\phi$ relies on the choice of some $i_0\in I'-J$, the above proposition shows it is the support of the cohomology sheaf of $F_{21}(\Theta_1(\si_J,\phi))$, which is independent of the choice of $i_0$.
\end{remark}

\subsubsection{Full and faithful functor}
Let $\si_{J_1}, \si_{J_2} \in \Si_1$, $\phi_1 \in M_{1,\si_{J_1}}$ and $\phi_2\in M_{1,\si_{J_2}}$, 
where $J_1,J_2\in\Lambda=\{J\subset \bar I \mid I'\not\subset J\}$. Define 
$$
c_{1,i}=\begin{cases}
\langle \phi_1, b_i \rangle_{\si_{J_1}},& i\subset J_1, \\ 
-\infty,& \text{otherwise};
\end{cases}
\quad 
c_{2,i}=\begin{cases}
\langle \phi_2, b_i \rangle_{\si_{J_2}},& i\subset J_2, \\ 
-\infty,& \text{otherwise.}
\end{cases}
$$
Define the polyhedral set $C(t_1,\dots, t_{n+1})\subset M_\bR$ to be
$$
C(t_1,\dots, t_{n+1}):=\{x\in M_\bR\mid \langle x,b_i\rangle > t_i\}.
$$
In the remainder of this subsection, we let $\bc_1$ and $\bc_2$ denote 
$(c_{1,1},\ldots,c_{1,n})$ and $(c_{2,1},\ldots,c_{2,n})$, respectively,
and write $\bt$ for $(t_1,\ldots,t_n)$.

It is obvious that $C(\bc_1, c_{1,n+1})=((\si_{J_1})_{ \phi_1}^\vee)^\circ$, and 
$C(\bc_2, c_{2,n+1})=((\si_{J_2})_{\phi_2}^{\vee})^\circ$. Furthermore, define $D(\bt, t_{n+1})$ to be 
$$
D(\bt, t_{n+1}):=\{x\in M_\bR\mid \langle x,b_i\rangle > t_i,\ i=1,\dots, n; 
\langle x,b_{n+1} \rangle \geq t_{n+1}\}.
$$

\begin{lemma}
\label{lem:hyperplane}
If $n+1\in J_1$, for any $\phi_1\in M_{1,\si_J}$  there is an $s_1 \in [ c_{1,n+1}, c_{1,n+1}+1)$ such that 
$D(\bc_1, s_1)\subset F((\si_{J_1})_{\phi_1}^\vee)^\circ$. The same result holds for $J_2$ as well.
\end{lemma}
\begin{proof}
Let $x=x_1b_1^*+\dots+x_{n+1} b^*_n\in D(\bc_1, s_1)$, for some 
$s_1=c_{1,n+1}+\epsilon$ where $\epsilon>0 $ will be determined 
below. 
Recall that in the definition of $F((\si_{J_1})_{\phi_1}^\vee)^\circ$, 
we have chosen an $i_0 \in I'-J_1$. Without the loss of generality, in this proof we assume $i_0=1$. Set 
$$
m_i=\begin{cases}\lceil x_i \rceil -1 -c_{1,i},& i\in J_1\cap I',\\ 
\lceil x_i \rceil -1 ,&  i\in I'-(J_1\cup \{1\},) 
\end{cases}
$$
and $\mathfrak{m}=(m_2,\ldots, m_{n'})\in \bZ^{n'-1}$. It suffices to show that $x\in ((\si_{J'_1})^\vee_{\gamma(\mathfrak{m})})^\circ$  after we specify a particular $\epsilon\in (0,1)$ (which depends on $\phi_1$ but not on $x$). The coordinate $x_1$ satisfies
\begin{eqnarray*}
x_1 &\geq& \frac{1}{\alpha_1}\bigl(s_1-\sum_{i=2}^{n'}\alpha_i x_i\bigr)\\
\   &\geq& \frac{1}{\alpha_1}\bigl(s_1-\sum_{i\in J_1\cap I'} (m_i+ c_{1,i}+1)\alpha_i - \sum_{i\in I'-(J_1\cup \{1\})}(m_i+1)\alpha_i\bigr)\\
\   &\geq& \frac{1}{\alpha_1}\bigl(c_{1,n+1}+(\epsilon-1)+1-\sum_{i\in J_1\cap I'} (m_i+ c_{1,i}+1)\alpha_i - \sum_{i\in I'-(J_1\cup \{1\})}(m_i+1)\alpha_i\bigr)\\
\   &\geq& \frac{1}{\alpha_1}\bigl(c_{1,n+1}+(\epsilon-1)+\alpha_1-\sum_{i\in J_1\cap I'} (m_i+ c_{1,i})\alpha_i - \sum_{i\in I'-(J_1\cup \{1\})}m_i\alpha_i \bigr).
\end{eqnarray*}
The last inequality depends on the fact $\alpha_1+\dots+\alpha_{n'}\leq 1$.  For any $\mathfrak{m}'=(m_2',\ldots,m_{n'}')\in \bZ^{n'-1}$, define
$$
u(\fm')=\frac{1}{\alpha_1}\bigl(c_{1,n+1}-\sum_{i\in J_1\cap I'} (m'_i+ c_{1,i})\alpha_i - \sum_{i\in I'-(J_1\cup \{1\})}m'_i\alpha_i\bigr).
$$ 
Since $\alpha_1,\ldots, \alpha_{n'}$ are rational numbers, 
$$
A(\phi_1) :=\{ u(\fm')+1-\lceil u(\fm')\rceil\mid \fm'\in \bZ^{n'-1}\}
$$
is a finite subset of $(0,1]$. Define
$$
\epsilon:= 1-\frac{\alpha_1}{2} \min A(\phi_1) \in (0,1).
$$
Then
$$
x_1\geq u(\fm)+1 + \frac{\epsilon-1}{\alpha_1} \geq  \lceil u(\fm)\rceil + \frac{1}{2}\min A(\phi_1) > \lceil u(\fm)\rceil,
$$
which implies that $x\in ((\si_{J'_1})^\vee_{\gamma(\mathfrak{m})})^\circ$.
\end{proof}

The lemma above implies the relation $D(\bc_1,s_1)\subset F((\si_{J_1})^\vee_{\phi_1})^\circ \subset C(\bc_1,c_{1,n+1})$. 
Moreover, given 
$$
x=x_1b_1^*+\dots+ x_n b^*_n\in F((\si_{J_1})^\vee_{\phi_1})^\circ - D(\bc_1,s_1)
$$ 
and any $l\in I'$, there is a unique 
$$
r_{J_1,\phi_1,l}(x) = x_1 b^*_1+\dots+ x_{l-1}b^*_{l-1}+\hat x_l b_l^*+x_{l+1}b_{l+1}^*+\dots+x_nb_n^*
$$ 
such that $\langle r_{J_1,\phi_1,l}(x) , b_{n+1} \rangle = s_1$, where $\hat x_l\geq x_l$. 
Meanwhile, given $$x=x_1b_1^*+\dots x_n b^*_n\in C(\bc_1,c_{1,n+1})-F((\si_{J_1})^\vee_{\phi_1})^\circ$$ and any $l\in I'$, 
there is also a unique 
$$r_{J_1,\phi_1,l}'= x_1b_1^*+\dots x_{l-1}b_{l-1}^*+\hat x'_l b_l^*+ x_{l+1}b_{l+1}^*+\dots+ x_n b_n^*$$ 
such that $\langle r_{J_1,\phi_1,l}'(x), b_{n+1} \rangle=c_{1,n+1}$, where $\hat x'_l \leq x_l$.

\begin{proposition}
\label{prop:III-hom-1}
Let $J_1$ and $J_2$ be two proper subsets of $\bar I$ such that $n+1\in J_1,J_2$. If $((\si_{J_1})^\vee_{\phi_1})^\circ\not\subset ((\si_{J_2})^\vee_{\phi_2})^\circ$, then $F((\si_{J_1})^\vee_{\phi_1})^\circ \not \subset F((\si_{J_2})^\vee_{\phi_2})^\circ$, and $F((\si_{J_1})^\vee_{\phi_1})^\circ-F((\si_{J_2})^\vee_{\phi_2})^\circ$ is a contractible set.
\end{proposition}

\begin{proof}
Since $((\si_{J_1})_{\phi_1}^\vee)^\circ \not \subset ((\si_{J_2})_{\phi_2}^{\vee})^\circ$, we have some $c_{1,l }< c_{2,l}$ for some $l$. Recall that in the definition of $F(\si_J)^\vee_\chi$ we have chosen an $i_0\in I'-J$. Here we fix $i_{1,0}\in I'-J_1$ and $i_{2,0} \in I'-J_2$. We prove the statement in the following two cases.

Case $l=n+1$: Let $s_1\in [c_{1,n+1}, c_{1,n+1}+1)$ be as given in Lemma \ref{lem:hyperplane}, so that $D(\bc_1,s_1)\subset F((\si_{J_1})_{\phi_1}^\vee)^\circ$. 
By definition, $F((\si_{J_2})_{\phi_2}^{\vee})^\circ\subset C(\bc_2, c_{2,n+1})$. Therefore, the non-empty set 
$$D(\bc_1,s_1)-C(\bc_2, c_{2,n+1})\subset F((\si_{J_1})_{\phi_1}^\vee)^\circ-F((\si_{J_2})_{\phi_2}^{\vee})^\circ.$$ 
We will show that there is a deformation retract $$h_t: (F((\si_{J_1})_{\phi_1}^\vee)^\circ-F((\si_{J_2})_{\phi_2}^\vee)^\circ)\times [0,1]\to F((\si_{J_1})_{\phi_1}^\vee)^\circ-F((\si_{J_2})_{\phi_2}^\vee)^\circ,$$
such that $h_0=id$ and the image of $h_1$ is inside $D(\bc_1,s_1)-C(\bc_2, c_{2, n+1})$, 
while $h_1$ is the identity map on $D(\bc_1,s_1)-C(\bc_2, c_{2, n+1})$. Given $x=x_1 b_1^* +\dots+ x_nb_n^*\in M_\bR$, the retract $h_t$ is defined as
$$
h_t(x)=\begin{cases}
tx+(1-t)r_{J_1,\phi_1,i_{2,0}}(x), &\text{if }x\in F((\si_{J_1})_{\phi_1}^\vee)^\circ-(D(\bc_1,s_1)\cup F((\si_{J_2})_{\phi_2}^\vee)^\circ)\\
x,& \text{if }x\in D(\bc_1,s_1)-C(\bc_2, c_{2,n+1})\\
tx+(1-t)r_{J_2,\phi_2,i_{1,0}}'(x), &\text{if }x\in (C(\bc_2, c_{2,n+1})-F((\si_{J_2})_{\phi_2}^{\vee})^\circ)\cap F((\si_{J_1})_{\phi_1}^\vee)^\circ.
\end{cases}
$$
Since the closures of $D(\bc_1,s_1)$ and $C(\bc_2, c_{2,n+1})$ are dual cones of toric cones in a fan, 
$D(\bc_1,s_1)-C(\bc_2, c_{2,n+1})$ is contractible. Hence we conclude that 
$F((\si_{J_1})_{\phi_1}^\vee)^\circ-F((\si_{J_2})_{\phi_2}^{\vee})^\circ$ is contractible.

Case $l\ne n+1$: In this case, one may assume that $c_{1,n+1}\geq c_{2,n+1}$ since otherwise we might let $l$ to be $n+1$ and goes back the the previous case. 
Similarly, we define a deformation retract 
$$
h_t: (F((\si_{J_1})_{\phi_1}^\vee)^\circ-F((\si_{J_2})_{\phi_2}^{\vee})^\circ)\times [0,1]\to F((\si_{J_1})_{\phi_1}^\vee)^\circ-F((\si_{J_2})_{\phi_2}^{\vee})^\circ,
$$ given as below.
$$
h_t(x)=\begin{cases}
tx+(1-t)r_{J_1,\phi_1,i_{2,0}}(x), &\text{if } x\in F((\si_{J_1})_{\phi_1}^\vee)^\circ-(D(\bc_1,s_1)\cup F((\si_{J_2})_{\phi_2}^{\vee})^\circ);\\
x, & \text{if }x\in D(\bc_1,s_1)-F((\si_{J_2})_{\phi_2}^{\vee})^\circ.
\end{cases}
$$
Since $C(\bc_2,s_2)\subset F((\si_{J_2})_{\phi_2}^{\vee})^\circ\subset C(\bc_2,c_{2,n+1})$, we have 
$$
D(\bc_1,s_1)-C(\bc_2,c_{2,n+1}) \subset D(\bc_1,s_1)-F((\si_{J_2})_{\phi_2}^{\vee})^\circ
\subset D(\bc_1,s_1)-C(\bc_2,s_1).
$$
The fact that $c_{1,n+1}\geq c_{2,n+1}$ implies $s_1\geq s_2\geq c_{2,n+1}$. Therefore 
$$
D(\bc_1,s_1)-C(\bc_2,c_{2,n+1})=D(\bc_1,s_1)-C(\bc_2,s_2),
$$
and then 
$$
D(\bc_1,s_1)-F((\si_{J_2})_{\phi_2}^{\vee})^\circ\ = D(\bc_1,s_1)-C(\bc_2,c_{2,n+1})
$$ 
is a non-empty contractible set. The above deformation retract shows that $F((\si_{J_1})_{\phi_1}^\vee)^\circ-F((\si_{J_2})_{\phi_2}^{\vee})^\circ$ is also a contractible set.
\end{proof}

\begin{proposition}
\label{prop:III-hom-2}
If $n+1\not \in J_1$ or $n+1\not\in J_2$, then 
$$((\si_{J_1})^\vee_{\phi_1})^\circ\not\subset ((\si_{J_2})^\vee_{\phi_2})^\circ \implies \text{$F((\si_{J_1})^\vee_{\phi_1})^\circ-F((\si_{J_2})^\vee_{\phi_2})^\circ$ is a non-empty contractible set}.$$
\end{proposition}
\begin{proof}
The proof is similar to Proposition \ref{prop:III-hom-1}. There are three cases.

Case $n+1\not \in J_1$, and $n+1\in J_2$: Notice that 
$F((\si_{J_1})_{\phi_1}^\vee)^\circ=C(\bc_1,c_{1,n+1}).$ Let $i_{1,0}\in I'-J_1$, since $I'\not \subset J_1$. Define the deformation retract between $F((\si_{J_1})^\vee_{\phi_1})^\circ-F((\si_{J_2})^\vee_{\phi_2})^\circ$ and $C(\bc_1,c_{1,n+1})-C(\bc_2,c_{2,n+1})$ as
$$
h_t(x)=\begin{cases}
x,& \text{if }x\in C(\bc_1,c_{1,n+1})-C(\bc_2, c_{2,n+1})\\
tx+(1-t)r_{J_2,\phi_2,i_{1,0}}'(x), & \text{if }x\in (C(\bc_2, c_{2,n+1})-F((\si_{J_2})_{\phi_2}^{\vee})^\circ)\cap C(\bc_1,c_{1,n+1}).
\end{cases}
$$

Case $n+1\in J_1$, and $n+1\not\in J_2$: $F((\si_{J_2})_{\phi_2}^\vee)^\circ=C(\bc_2,c_{2,n+1}).$ Let $i_{2,0}\in I'-J_2$. 
Define the deformation retract between $F((\si_{J_1})^\vee_{\phi_1})^\circ-F((\si_{J_2})^\vee_{\phi_2})^\circ$ and 
$D(\bc_1,s_1)-C(\bc_2,c_{2,n+1})$ as
$$
h_t(x)=\begin{cases}
tx+(1-t)r_{J_1,\phi_1,i_{2,0}}(x), & \text{if }x\in F((\si_{J_1})_{\phi_1}^\vee)^\circ-
(D(\bc_1,s_1)\cup C(\bc_2, c_{2,n+1}))\\
x,&  \text{if }x\in D(\bc_1,s_1)-C(\bc_2, c_{2,n+1}).
\end{cases}
$$

Case $n+1\not\in J_1$ and $n+1\not\in J_2$: 
This is trivial since $F((\si_{J_2})_{\phi_2}^\vee)^\circ=C(\bc_2,c_{2,n+1})$ and $F((\si_{J_1})_{\phi_1}^\vee)^\circ=C(\bc_1,c_{1,n+1})$.

\end{proof}

\begin{theorem}
The functors $F_{21}'$ and $F_{21}$ are quasi-embeddings. If restricted on $\Coh_T(\cX_1)$, $F_{21}'$ is a quasi-embedding of $\Coh_T(\cX_1)$ into $\Coh_T(\cX_2)$.
\end{theorem}
\begin{proof}
One only needs to show that $F_{21}$ is a cohomologically full and faithful functor. Since we have 
$$
\Ext^*(\Theta_1(\si_{J_1},\phi_1),\Theta_1(\si_{J_2},\phi_2)) =
\begin{cases}
\bC[0] &\text{if $((\si_{J_1})_{\phi_1}^\vee)^\circ \subset ((\si_{J_2})_{\phi_2}^{\vee})^\circ$},\\
0 &\text{if $((\si_{J_1})_{\phi_1}^\vee)^\circ \not\subset ((\si_{J_2})_{\phi_2}^{\vee})^\circ$},
\end{cases}
$$
and
\begin{eqnarray*}
&&
\Ext^*(i_{F((\si_{J_1})_{\phi_1}^\vee)^\circ!}F((\si_{J_1})_{\phi_1}^\vee)^\circ,i_{F((\si_{J_2})_{\phi_2}^\vee)^\circ!}F((\si_{J_2})_{\phi_2}^\vee)^\circ\\
&=&\begin{cases}
\bC[0] & \text{if $F((\si_{J_1})_{\phi_1}^\vee)^\circ \subset F((\si_{J_2})_{\phi_2}^{\vee})^\circ$,}\\
\\
0 &\text{\parbox[t]{10cm}{$F((\si_{J_1})_{\phi_1}^\vee)^\circ \not\subset F((\si_{J_2})_{\phi_2}^{\vee})^\circ$ and $F((\si_{J_1})^\vee_{\phi_1})^\circ-F((\si_{J_2})^\vee_{\phi_2})^\circ$ is contractible,}}
\end{cases}
\end{eqnarray*}
the desired result follows immediately from Proposition \ref{prop:III-F}, Proposition \ref{prop:III-hom-1} and Proposition \ref{prop:III-hom-2}, and the simple fact that
$$((\si_{J_1})_{\phi_1}^\vee)^\circ \subset ((\si_{J_2})_{\phi_2}^{\vee})^\circ \implies F((\si_{J_1})_{\phi_1}^\vee)^\circ \subset F((\si_{J_2})_{\phi_2}^{\vee})^\circ.$$ 
\end{proof}

\begin{figure}[h]
\psfrag{F}{\small $F_{21}$}
\includegraphics[scale=0.5]{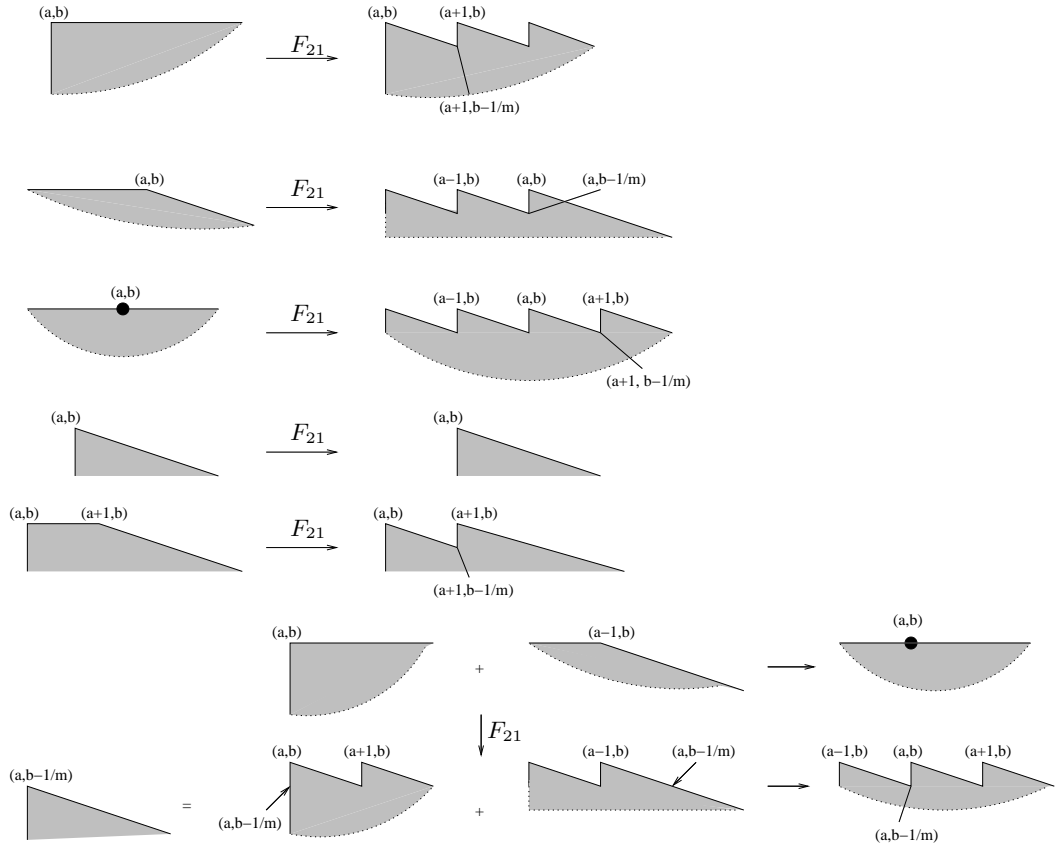}
\caption{$a$ and $b$ are integers. The constructible sheaves are costandard sheaves over the shaded regions.}
\label{fig:III}
\end{figure}

\begin{example}
$N=\bZ^2$, 
$$
\beta_1=\left[ \begin{array}{rrr}  1& -1 & 0\\ 0 & -m & -1 \end{array} \right],\quad
\beta_2=\left[ \begin{array}{rr}  1& -1\\ 0 & -m \end{array} \right],\quad
\beta'=\left[\begin{array}{rrr} 1& -1 & 0\\ 0 & -m & -m \end{array} \right].
$$
$\cX_1$ is the total space of $\cO_{\bP^1}(-m)$, and
$\cX_2=[\bC^2/\bZ_m]$.
\begin{eqnarray*}
&& v_1=b_1=(1,0), \quad v_2=b_2= (-1,-m),\quad v_3=b_3= (0,-1), \\
&& r_1=r_2=r_3=1,\quad r_3'=m,\quad a_1=a_2=\alpha_1=\alpha_2=\frac{1}{m}, \quad a_1'= a_2'=\beta_1=\beta_2=1.
\end{eqnarray*}
$$
\frac{a_1}{r_1}+\frac{a_2}{r_2} \leq \frac{1}{r_3}\Leftrightarrow m\geq 2.
$$

The Fourier-Mukai functors for constructible sheaves are shown in Figure \ref{fig:III}.
\end{example}

\end{document}